\documentclass[12pt,a4paper]{article}
\usepackage[pdfencoding=auto]{hyperref}
\usepackage{bookmark}
\usepackage[utf8]{inputenc}
\usepackage{amsmath}
\usepackage{amsfonts}
\usepackage{lmodern}
\usepackage{amssymb}
\usepackage{amsthm}
\usepackage{mathtools}
\usepackage{url}
\usepackage[english]{babel}
\usepackage{tikz}
\usepackage{graphicx,color}

\usepackage[top=2cm, bottom=2cm, left=2.5cm, right=2.5cm]{geometry}
\theoremstyle{plain}% default
\newtheorem{thm}{Theorem}[section]
\newtheorem{lem}[thm]{Lemma}
\newtheorem{prop}[thm]{Proposition}
\newtheorem{cor}[thm]{Corollary}

\newtheorem{claim}{Claim}

\theoremstyle{definition}
\newtheorem{defn}{Definition}[section]

\theoremstyle{remark}
\newtheorem{rem}{Remark}[section]

\newcommand\Tr{\operatorname{Tr}}

\newcommand\Cov{\operatorname{Cov}}
\newcommand\Capa{\operatorname{Cap}}

\newcommand\thr{\alpha_c}
\newcommand\thrcz{\alpha_{\#}}
\newcommand\thran{\alpha_{1}}

\allowdisplaybreaks

\begin{document}
\hypersetup{pageanchor=false}
\pgfsetxvec{\pgfpoint{8pt}{0}}
\pgfsetyvec{\pgfpoint{0}{8pt}}
\title{Phase transition in loop percolation}
\author{Yinshan Chang and Art{\"e}m Sapozhnikov\\
MPI MSI,\\
04103 Leipzig, Germany}
\maketitle

\begin{abstract}
We are interested in the clusters formed by a Poisson ensemble of Markovian loops on infinite graphs. 
This model was introduced and studied in \cite{LejanMR2971372} and \cite{LemaireLeJan}. 
It is a model with long range correlations with two parameters $\alpha$ and $\kappa$. 
The non-negative parameter $\alpha$ measures the amount of loops, 
and $\kappa$ plays the role of killing on vertices penalizing ($\kappa>0$) or favoring ($\kappa<0$) appearance of large loops.
It was shown in \cite{LemaireLeJan} that for any fixed $\kappa$ and large enough $\alpha$, 
there exists an infinite cluster in the loop percolation on $\mathbb{Z}^d$.
In the present article, we show a non-trivial phase transition on the integer lattice $\mathbb{Z}^d$ ($d\geq 3$) for $\kappa=0$. 
More precisely, we show that there is no loop percolation for $\kappa=0$ and $\alpha$ small enough. 
Interestingly, we observe a critical like behavior on the whole sub-critical domain of $\alpha$, 
namely, for $\kappa=0$ and any sub-critical value of $\alpha$, the probability of one-arm event decays at most polynomially. 

For $d\geq 5$, we prove that there exists a non-trivial threshold for the finiteness of the expected cluster size. 
For $\alpha$ below this threshold, we calculate, up to a constant factor, the decay of the probability of one-arm event, two point function, 
and the tail distribution of the cluster size. 
These rates are comparable with the ones obtained from a single large loop and only depend on the dimension. 

For $d=3$ or $4$, we give better lower bounds on the decay of the probability of one-arm event, 
which show importance of small loops for long connections. 
In addition, we show that the one-arm exponent in dimension $3$ depends on the intensity $\alpha$.

\end{abstract}
\section{Introduction}
Consider an unweighted undirected graph $G=(V,E)$ and a random walk $(X_m,m\geq 0)$ on it with transition matrix $Q$.
Unless specified, we will assume that $(X_m,m\geq 0)$ is a simple random walk (SRW) on $\mathbb Z^d$. 

\medskip

As in \cite{LemaireLeJan}, an element $\dot\ell = (x_1,\dots,x_n)$ of $V^n$, $n\geq 2$, satisfying $x_1\neq x_2,\dots,x_n\neq x_1$ is called a non-trivial discrete based loop. 
Two based loops of length $n$ are equivalent if they coincide after a circular permutation of their coefficients, i.e., 
$(x_1,\dots,x_n)$ is equivalent to $(x_i,\dots,x_n,x_1,\dots,x_{i-1})$ for all $i$. 
Equivalence classes of non-trivial discrete based loops for this equivalence relation are called (non-trivial) discrete loops. 

Given an additional parameter $\kappa>-1$, we associate to each based loop $\dot\ell=(x_1,\dots,x_n)$ the weight
\begin{equation}\label{def:dotmu}
\dot\mu_\kappa(\dot\ell)=\frac{1}{n}\left(\frac{1}{1+\kappa}\right)^{n}Q^{x_1}_{x_2}\cdots Q^{x_{n-1}}_{x_n}Q^{x_n}_{x_1}.
\end{equation}
The push-forward of $\dot\mu_\kappa$ on the space of discrete loops is denoted by $\mu_\kappa$. 
(Note that our parameter $\kappa$ in \eqref{def:dotmu} corresponds to $\frac{\kappa}{2d}$ in \cite{LemaireLeJan}.)

For $\alpha>0$ and $\kappa>-1$, let $\mathcal L_{\alpha,\kappa}$ be the Poisson loop ensemble of intensity $\alpha\mu_\kappa$, 
i.e, $\mathcal L_{\alpha,\kappa}$ is a random countable collection of discrete loops such that the point measure $\sum\limits_{\ell\in \mathcal{L}_{\alpha,\kappa}}\delta_{\ell}$ is 
a Poisson random measure of intensity $\alpha\mu_\kappa$. (Here, $\delta_\ell$ means the Dirac mass at the loop $\ell$.)
The collection $\mathcal L_{\alpha,\kappa}$ is induced by the Poisson ensemble of non-trivial continuous loops defined by Le Jan \cite{LejanMR2971372}. 

\medskip

The Poisson ensembles of Markovian loops were introduced informally by Symanzik \cite{Symanzik}.
They have been rigorously defined and studied by Lawler and Werner \cite{LawlerWernerMR2045953} in the context of two dimensional Brownian motion 
(the Brownian loop soup). 
The random walk loop soup on graphs was studied by Lawler and Limic \cite[Chapter 9]{LawlerMR2677157}, 
and its convergence to the Brownian loop soup by Ferreras and Lawler \cite{LawlerMR2255196}. 
Extensive investigation of the loop soup on finite and infinite graphs was done by Le Jan \cite{LeJanMR2675000,loop} 
for reversible Markov processes, and by Sznitman \cite{SznitmanMR2932978} in the context of reversible Markov chains on finite graphs 
from the point of view of occupation field and relation with random interlacement. 
A comprehensive study of Poisson ensembles of loops of one-dimensional diffusions was done by Lupu \cite{Lupuonedim}.
Let us also mention the works of Sheffield and Werner \cite{SheffieldWernerMR2979861} and Camia \cite{Camia}, 
who studied clusters in the two dimensional Brownian loop soup. 

\medskip

In this paper we are interested in percolative properties of clusters formed by $\mathcal L_{\alpha,\kappa}$ on $\mathbb Z^d$, motivated by the work of Le Jan and Lemaire \cite{LemaireLeJan}. 
An edge $e\in E$ is called open for $\mathcal L_{\alpha,\kappa}$ if it is traversed by at least one loop from $\mathcal L_{\alpha,\kappa}$. 
Maximal connected components of open edges for $\mathcal L_{\alpha,\kappa}$ form open clusters $\mathcal C_{\alpha,\kappa}$ of vertices. 
The percolation probability is defined as 
\[
\theta(\alpha,\kappa) \stackrel{\mathrm{def}}= \mathbb P[\#\mathcal C_{\alpha,\kappa}(0) = \infty].
\]
It is known that $\theta(\alpha,\kappa)$ is an increasing function of $\alpha$ and a decreasing function of $\kappa$, see \cite[Proposition~4.3]{LemaireLeJan}. 
In particular, if we define the critical thresholds
\[
\thr(\kappa)=\inf\{\alpha\geq 0:\theta(\alpha,\kappa)>0\}\text{ and }\kappa_c(\alpha)=\inf\{\kappa\geq -1:\theta(\alpha,\kappa)=0\},
\]
then $\theta(\alpha,\kappa) = 0$ for $\alpha<\thr(\kappa)$ or $\kappa>\kappa_c(\alpha)$, 
and $\theta(\alpha,\kappa) > 0$ for $\alpha>\thr(\kappa)$ or $\kappa<\kappa_c(\alpha)$, 
$\thr(\kappa)$ is non-decreasing in $\kappa$ and $\kappa_c(\alpha)$ is non-decreasing in $\alpha$. 
The first properties of these thresholds were proved by Le Jan and Lemaire \cite[Proposition~4.3, Remark~4.4]{LemaireLeJan}: 
\begin{itemize}\itemsep0pt
\item 
for fixed $\kappa$, $\thr(\kappa)<\infty$, i.e., for large enough $\alpha$ there is an infinite open cluster,
\item
for fixed $\alpha$, $\kappa_c(\alpha)<\infty$, i.e., there is no infinite open cluster for large enough $\kappa$,
\item
on $\mathbb Z^2$, for any $\alpha$, $\kappa_c(\alpha)\geq 0$. 
\end{itemize}
This picture can be complemented with the following result. 
\begin{thm}\label{thm:phase transition}
For simple random walk loop percolation on $\mathbb Z^d$, $d\geq 3$, 
\begin{itemize}\itemsep0pt
\item[(1)]
$\thr \stackrel{def}= \thr(0)>0$,  
\item[(2)]
for any $\alpha>0$, $\kappa_c(\alpha)\geq 0$, 
\end{itemize}
i.e., the percolation phase transition is non-trivial, see Figure~\ref{fig:critical curve}.

In contrast, for any connected recurrent\footnote{A graph $G$ is called recurrent is the simple random walk on $G$ is recurrent.} graph $G$, 
for $\kappa=0$ and $\alpha>0$, with probability $1$, all the vertices are in the same open cluster.
\end{thm}

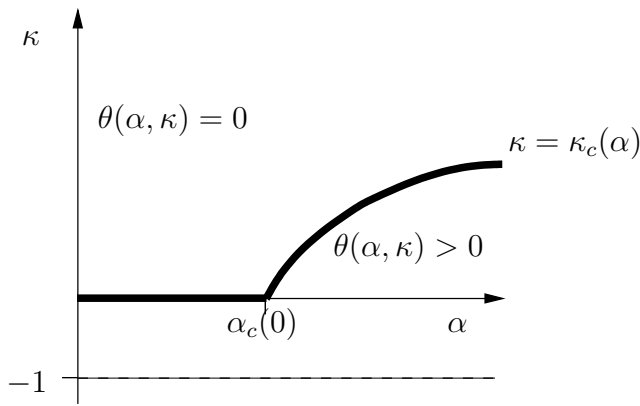
\begin{figure}[!htbp]
\begin{picture}(0,0)%
\includegraphics{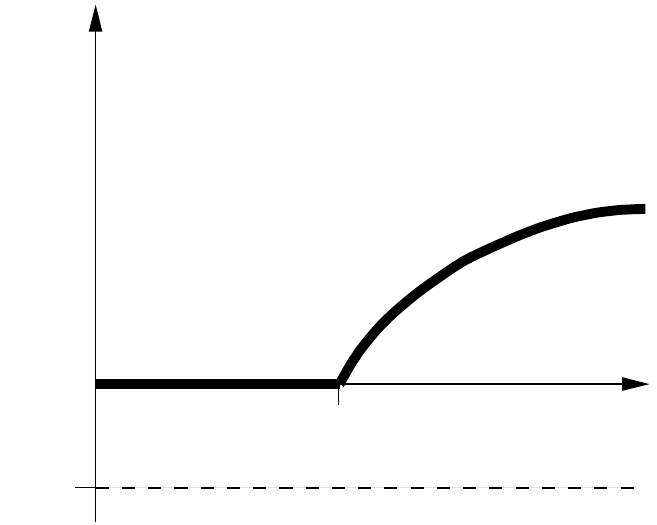}%
\end{picture}%
\setlength{\unitlength}{4144sp}%
\begingroup\makeatletter\ifx\SetFigFont\undefined%
\gdef\SetFigFont#1#2#3#4#5{%
  \reset@font\fontsize{#1}{#2pt}%
  \fontfamily{#3}\fontseries{#4}\fontshape{#5}%
  \selectfont}%
\fi\endgroup%
\begin{picture}(3000,2389)(2731,-4238)
\put(4051,-3796){\makebox(0,0)[lb]{\smash{{\SetFigFont{12}{14.4}{\familydefault}{\mddefault}{\updefault}{\color[rgb]{0,0,0}$\thr(0)$}%
}}}}
\put(2746,-4156){\makebox(0,0)[lb]{\smash{{\SetFigFont{12}{14.4}{\familydefault}{\mddefault}{\updefault}{\color[rgb]{0,0,0}$-1$}%
}}}}
\put(3286,-2581){\makebox(0,0)[lb]{\smash{{\SetFigFont{12}{14.4}{\familydefault}{\mddefault}{\updefault}{\color[rgb]{0,0,0}$\theta(\alpha,\kappa)=0$}%
}}}}
\put(2836,-2086){\makebox(0,0)[lb]{\smash{{\SetFigFont{12}{14.4}{\familydefault}{\mddefault}{\updefault}{\color[rgb]{0,0,0}$\kappa$}%
}}}}
\put(5356,-3796){\makebox(0,0)[lb]{\smash{{\SetFigFont{12}{14.4}{\familydefault}{\mddefault}{\updefault}{\color[rgb]{0,0,0}$\alpha$}%
}}}}
\put(4681,-3346){\makebox(0,0)[lb]{\smash{{\SetFigFont{12}{14.4}{\familydefault}{\mddefault}{\updefault}{\color[rgb]{0,0,0}$\theta(\alpha,\kappa)>0$}%
}}}}
\put(5716,-2716){\makebox(0,0)[lb]{\smash{{\SetFigFont{12}{14.4}{\familydefault}{\mddefault}{\updefault}{\color[rgb]{0,0,0}$\kappa=\kappa_c(\alpha)$}%
}}}}
\end{picture}%
\caption{Illustration of critical curve $\alpha\mapsto\kappa_{c}(\alpha)$ in $\mathbb{Z}^d$ ($d\geq 3$).}
\label{fig:critical curve}
\end{figure}

The first statement of Theorem~\ref{thm:phase transition} will directly follow from further stronger results of 
Theorem~\ref{thm: polynomial decay} that for some positive value of $\alpha$, the one arm probability tends to $0$ polynomially, 
the second statement will be proved in Proposition~\ref{prop: existence of percolation for negative kappa on Zd}, 
and the third in Proposition~\ref{prop: existence of percolation for recurrent graph and zero kappa}. 
We should mention that during the write up of this paper, Titus Lupu posted a paper \cite{Lupu} in which he proves that 
for the loop percolation on $\mathbb Z^d$, $d\geq 3$, $\thr\geq \frac 12$ 
using a new coupling between the loop percolation and the Gaussian free field. 
Later in Theorem~\ref{thm: asymptotic of critical value in high dimensions} we provide 
an asymptotic expression for $\thr$ as the dimension $d\to\infty$. 

\bigskip

The loop percolation on $\mathbb Z^d$, $d\geq 3$, has long range correlations, see Proposition~\ref{prop: long range correlation},
it is translation invariant and ergodic with respect to the lattice shifts, see Proposition~\ref{prop: translation invariance and ergodicity},
it satisfies the positive finite energy property, and thus there can be at most one infinite open cluster, see Proposition~\ref{prop: uniqueness of infinite cluster}.
It is worth to make a comparison with other percolation models with long range correlations, which have been recently actively studied, 
for instance, the vacant set of random interlacements \cite{SznitmanMR2680403} or the level sets of the Gaussian free field \cite{RodriguezSznitmanMR3053773}. 
As we will soon see from our main results, the loop percolation displays rather different behavior in the sub-critical regime than 
the above mentioned models. The decay of the one-arm probability in the loop percolation is at most polynomial, and in the other models 
it is exponential or stretched exponential, see \cite{PopovTeixeira,PopovRath}. 
As a consequence, the so-called decoupling inequalities \cite{SznitmanMR2891880,PopovTeixeira,PopovRath}, which are a powerful tool 
in studying those models, are not valid for the loop percolation.

\bigskip

For $x,y\in V$, we write $x\stackrel{\mathcal L_{\alpha,\kappa}}\longleftrightarrow y$ if $y\in \mathcal C_{\alpha,\kappa}(x)$ 
(equivalently, $x\in\mathcal C_{\alpha,\kappa}(y)$).
For two vertex sets $A$ and $B$, the notation $A\overset{\mathcal{L}_{\alpha,\kappa}}{\longleftrightarrow} B$ means that 
there exist $x\in A$ and $y\in B$ such that $x\overset{\mathcal{L}_{\alpha,\kappa}}{\longleftrightarrow} y$. 
For a vertex $x$ and a set of vertices $B$, we write $x\overset{\mathcal{L}_{\alpha,\kappa}}{\longleftrightarrow} B$ instead of 
$\{x\}\overset{\mathcal{L}_{\alpha,\kappa}}{\longleftrightarrow} B$.

Our main object of interest in this paper is the one arm probability for the loop percolation on $\mathbb{Z}^d$ with $\alpha<\thr$ and $\kappa=0$:
\[
\mathbb{P}\left[0\overset{\mathcal{L}_{\alpha,0}}{\longleftrightarrow}\partial B(0,n)\right],
\]
where $\partial B(0,n)$ is the boundary of the box of side length $2n$ centered at $0$. 
A general lower bound for $\alpha>0$ can be obtained by calculating the probability of connection by one big loop in $\mathcal L_{\alpha,0}$:
\begin{thm}\label{thm:at most polynomial decay of the connectivity}
For $d\geq 3$, $\alpha>0$ and $\kappa=0$, there exists $c = c(d)>0$ such that for all $n\geq 1$,
$$\mathbb{P}\left[0\overset{\mathcal{L}_{\alpha,0}}{\longleftrightarrow}\partial B(0,n)\right]\geq 
\mathbb P\left[\exists \ell\in\mathcal L_{\alpha,0}~:~ \text{$\ell$ intersects both $0$ and $\partial B(0,n)$}\right]\geq
\alpha\cdot c\cdot n^{2-d}.$$
\end{thm}
It is interesting that we have in the same model a non-trivial phase transition together with an at most polynomial decay of one-arm connectivity for sub-critical domain. 
This critical-like behavior might be understood in the following way. 
By the second statement of Theorem~\ref{thm:phase transition}, for any $\alpha>0$, $\kappa_c(\alpha)\geq 0$. 
Thus, $]0,\thr[\times \{0\}$ is a part of the critical curve of $(\alpha,\kappa)$, see Figure~\ref{fig:critical curve}.
This polynomial decay phenomenon appears exactly at the critical value $0$ of the parameter $\kappa$.

\medskip

It is natural to consider whether $n^{2-d}$ is the right order of $\mathbb{P}[0\overset{\mathcal{L}_{\alpha,0}}{\longleftrightarrow}\partial B(0,n)]$. 
We cannot give an answer for the whole sub-critical domain. 
We introduce an auxiliary parameter as follows: for $d\geq 3$, 
\[
\thran = \thran(d) \overset{\text{def}}{=}
\sup\limits_{\beta>1}\inf\{\alpha>0:\limsup\limits_{n\rightarrow\infty}\mathbb{P}[B(0,n)\overset{\mathcal{L}_{\alpha,0}}{\longleftrightarrow} \partial B(0,\lceil\beta n\rceil)]=1\}.
\]
Our first step is the following polynomial upper bound:
\begin{thm}\label{thm: polynomial decay}\ 
\begin{itemize}\itemsep0pt
 \item[a)] For $d\geq 3$, $\thran>0$.
 \item[b)] For $d\geq 3$ and $\alpha<\thran$, there exist constants $C(d,\alpha)<\infty$ and $c(d,\alpha)>0$ such that 
$\lim\limits_{\alpha\rightarrow 0}c(d,\alpha)=d-2$ and for $N\geq 1$,
$$\mathbb{P}[0\overset{\mathcal{L}_{\alpha,0}}{\longleftrightarrow} \partial B(0,N)]\leq C(d,\alpha)\cdot N^{- c(d,\alpha)}.$$
In fact, one can take $c(d,\alpha) = d-2 - C(d)(\log\frac{1}{\alpha})^{-1}$ as $\alpha\to 0$, where $C(d)$ is a large enough constant.
\end{itemize}
\end{thm}
By Theorem~\ref{thm:at most polynomial decay of the connectivity}, $c(d,\alpha)\leq d-2$, and 
Theorem~\ref{thm: polynomial decay} suggests that $d-2$ could probably be the right exponent for the one-arm decay. 
This is indeed the case when the expected cluster size is finite, see Theorem~\ref{thm:exponent of one-arm connectivity for d larger than 5}. 
To state the result, we introduce another auxiliary parameter corresponding to the finiteness of expected cluster size: 
\begin{equation}
\thrcz = \thrcz(d)\overset{\text{def}}{=}\sup\{\alpha>0:\mathbb{E}[\#\mathcal{C}_{\alpha,0}(0)]<\infty\}.
\end{equation}
Our next result provides the strict positivity of $\thrcz$ and the order of one-arm decay for $d\geq 5$ together with the order of two point connectivity, 
the tail of cluster size and comparison between $\thrcz$ and $\thran$:
\begin{thm}\label{thm:exponent of one-arm connectivity for d larger than 5}\ 
\begin{itemize}\itemsep0pt
 \item[a)] For $d\geq 5$, $\thrcz>0$, and for $d=3$ or $4$, $\thrcz=0$.
 \item[b)] For $d\geq 5$ and $\alpha<\thrcz$, there exist constants $0<c(d,\alpha)<C(d,\alpha)<\infty$ such that for all $n$,
$$c(d,\alpha)n^{2-d}\leq\mathbb{P}[0\overset{\mathcal{L}_{\alpha,0}}{\longleftrightarrow}\partial B(0,n)]\leq C(d,\alpha)n^{2-d}.$$
 \item[c)] For $d\geq 5$ and $\alpha<\thrcz$, there exist $0<c(d,\alpha)<C(d,\alpha)<\infty$ such that for all $x\in\mathbb{Z}^d$,
 $$c(d,\alpha) (||x||_{\infty}+1)^{2(2-d)}\leq\mathbb{P}[x\in\mathcal{C}_{\alpha,0}(0)]\leq C(d,\alpha) (||x||_{\infty}+1)^{2(2-d)}.$$
 \item[d)] For $d\geq 5$ and $\alpha<\thrcz$, there exist $0<c(d,\alpha)<C(d,\alpha)<\infty$ such that for all $n$,
 \begin{equation*}
 c(d,\alpha)n^{1-d/2}\leq\mathbb{P}[\#\mathcal{C}_{\alpha,0}(0)>n]\leq C(d,\alpha)n^{1-d/2},
 \end{equation*}
 \item[e)] For $d\geq 5$, $\thrcz\leq\thran$.
\end{itemize}
\end{thm}
Theorems~\ref{thm:at most polynomial decay of the connectivity} and \ref{thm:exponent of one-arm connectivity for d larger than 5} 
suggest the following picture for sub-critical loop percolation in dimension $d\geq 5$: the large cluster typically contains a macroscopic loop of 
diameter comparable with the diameter of the cluster. 

This scenario however cannot be true for sub-critical loop percolation in dimensions $d = 3,4$, 
as we can get better lower bounds on the one-arm probability. 
In dimension $d=3$, we prove that $d-2$ is not the right exponent for the one-arm probability:
\begin{thm}\label{thm: d-2 is not the right exponent for d=3}
 For $d=3$, for $\alpha>0$, there exist $\epsilon(\alpha), c(\alpha)>0$ such that for all $n$,
 $$\mathbb{P}[0\overset{\mathcal{L}_{\alpha,0}}{\longleftrightarrow}\partial B(0,n)]\geq c(\alpha)n^{-1+\epsilon(\alpha)}.$$
\end{thm}
Note that $\lim\limits_{\alpha\rightarrow 0}\epsilon(\alpha)=0$ by Theorem~\ref{thm: polynomial decay}. 

In dimension $d=4$, we get an improved lower bound for the one-arm probability, still with exponent $d-2$, but with an extra logarithmic correction:
\begin{thm}\label{thm: refined lower bound for d=4}
 For $d=4$, there exist $\epsilon(\alpha),c(\alpha)>0$ such that
 $$\mathbb{P}[0\overset{\mathcal{L}_{\alpha,0}}{\longleftrightarrow}\partial B(0,n)]\geq c(\alpha)n^{-2}(\log n)^{\epsilon(\alpha)}.$$
\end{thm}
We conjecture that for the sub-critical loop percolation in dimension $d=4$, 
an upper bound on the one-arm probability is in similar form with logarithmic correction.

The results of Theorems~\ref{thm: d-2 is not the right exponent for d=3} and \ref{thm: refined lower bound for d=4} 
imply that the structure of connectivities in sub-critical loop percolation in dimensions $d=3,4$ 
is different from that in dimensions $d\geq 5$: macroscopic loops are not essential for 
formation of large connected components. 

\bigskip

All the upper bounds that we obtain hold either for $\alpha<\thrcz$ or for $\alpha<\thran$, 
i.e., for subregimes of the sub-critical phase. 
We expect that for $d\geq 3$, $\thrcz = \thr$, and for $d\geq 5$, $\thran = \thr$, 
but we do not have a proof yet. However, we can show that asymptotically, as $d\to\infty$, all these thresholds coincide. 
\begin{thm}\label{thm: asymptotic of critical value in high dimensions}
Asymptotically, as $d\to\infty$,
\[
2d - 6 + O(d^{-1}) \leq \thrcz\leq \thr \leq 2d + \frac 32 + O(d^{-1}).
\]
\end{thm}

\bigskip

{\it Outline of the paper.}
In the next section, we introduce the commonly used notation and collect some preliminary results about simple random walk on $\mathbb Z^d$ and some properties of the loop measure $\mu$. 
In Section~\ref{sec: basic properties of loop percolation}, 
we prove some elementary properties of the loop percolation on $\mathbb Z^d$, 
such as long-range correlations, translation invariance and ergodicity, the uniqueness of the infinite cluster, and the connectedness for $\kappa<0$. 
Except for the translation invariance, these properties will not be used in the proofs of the main results. 
In Section~\ref{sec: first results of one-arm connectivity}, 
we prove Theorems~\ref{thm:at most polynomial decay of the connectivity} and \ref{thm: polynomial decay}.
Finer results for the loop percolation in dimensions $d\geq 5$ are presented in Section~\ref{sec: dimension larger than 5}. 
In particular, the first $5$ subsections are devoted to the proof of Theorem~\ref{thm:exponent of one-arm connectivity for d larger than 5}, 
which is split into $5$ Propositions~\ref{prop:posivity of alpha *}, \ref{prop:exponent for one-arm connectivity for d larger than 5}, 
\ref{prop:two point connectivity}, \ref{prop:exponent of size of cluster for d larger than 5}, and \ref{prop: alpha * is smaller than alpha **}, 
and the last subsection contains the proof of Theorem~\ref{thm: asymptotic of critical value in high dimensions}. 
The proofs of Theorems~\ref{thm: d-2 is not the right exponent for d=3} and \ref{thm: refined lower bound for d=4} 
(refined lower bounds in dimension $d=3$ or $4$) are given in Section~\ref{sec: dimension three or four}. 
In Section~\ref{sec: general graphs}, we collect some results for the loop percolation on general graphs, 
such as triviality of the tail sigma-algebra, connectedness in recurrent graphs, and the continuity of $\kappa_c(\alpha)$. 
We finish the paper with an overview of some open questions. 

\section{Notation and preliminary results}

\subsection{Notation}
Let $G = (V,E)$ be an unweighted undirected graph. 
For $F\subseteq V$, let $\partial F = \{x\in F:\exists y\in V\setminus F\text{ such that }\{x,y\}\text{ is an edge}\}$ be the boundary of $F$. 

Let $(X_n,n\geq 0)$ be a simple random walk (SRW) on $G$. Let $\mathbb P^x$ be the law of SRW started from $x\in V$. 
Let $(G(x,y))_{x,y\in V}$ be the Green function for $(X_n,n\geq 0)$. 

For $F\subseteq V$, let $\tau(F)$ be the entrance time of $F$ and $\tau^{+}(F)$ be the hitting time of $F$ by $(X_n,n\geq 0)$:
\[
\tau(F)=\inf\{n\geq 0:X_n\in F\}\text{ and }\tau^{+}(F)=\inf\{n\geq 1:X_n\in F\}.
\]
For $x\in V$, we use the notation $\tau(x)$ and $\tau^{+}(x)$ instead of $\tau(\{x\})$ and $\tau^{+}(\{x\})$. 
\begin{defn}\label{defn:capacity}
The capacity of a set $F$ is defined by
$$\Capa(F)=\sum\limits_{x\in\partial F}\mathbb{P}^x[\tau^{+}(F)=\infty].$$
\end{defn}
For finite subsets of vertices of a transient graph, the capacity is positive and monotone, i.e., for any finite $F\subset F'\subset V$, 
$0<\Capa(F)\leq \Capa(F')$. 

\medskip

For $x\in V$ and a loop $\ell$, we write $x\in \ell$ if $\ell$ visits $x$, i.e., 
for some based loop $\dot\ell$ in the equivalence class $\ell$, $\dot\ell = (x_1,\dots,x_n)$ with $x_1 = x$. 
For $F\subseteq V$, we write $\ell\cap F\neq \phi$ if $\ell$ visits at least one vertex in $F$, 
and $\ell\subset F$ if all the vertices visited by $\ell$ are contained in $F$. 
For two sets of vertices $F_1$ and $F_2$, we write $F_1\stackrel{\ell}\longleftrightarrow F_2$ if 
the loop $\ell$ intersects both $F_1$ and $F_2$. If some of the two sets is a singleton, say $\{x\}$, then 
we omit the brackets from the notation. For instance, $x\stackrel{\ell}\longleftrightarrow y$ 
means that the loop $\ell$ intersects both $x$ and $y$.

For $F\subseteq V$, $\alpha>0$ and $\kappa>-1$, we write 
\[
(\mathcal{L}_{\alpha,\kappa})_F\overset{\text{def}}{=}\{\ell\in\mathcal{L}_{\alpha,\kappa}~:~\ell\cap F\neq\phi\},
\]
\[
(\mathcal{L}_{\alpha,\kappa})^F\overset{\text{def}}{=}\{\ell\in\mathcal{L}_{\alpha,\kappa}~:~\ell\subset F\}.
\]

\medskip

Since most of the time we will deal with the case $\kappa = 0$, we accept the following convention:
\begin{center}
\emph{In case $\kappa=0$, we omit the subindex ``$\kappa$'' from all the notation.}
\end{center}
For instance, we will write $\mu = \mu_0$, $\mathcal L_{\alpha} = \mathcal L_{\alpha,0}$, $\mathcal C_{\alpha} = \mathcal C_{\alpha,0}$. 

\medskip

Throughout the following context, we denote by $M_p^{+}$ the set of $\sigma$-finite point measures on the space of discrete loops on $G$, 
and by $\mathcal F$ the canonical $\sigma$-algebra on $M_p^+$. 
For $K\subseteq V$ and a point measure $m=\sum\limits_{i\in\mathbb{N}}c_i\delta_{\ell_i}$ of loops 
where $\delta_{\ell_i}$ is the Dirac mass at the loop $\ell_i$, 
define $m_K=\sum\limits_{i\in\mathbb{N}}c_i\delta_{\ell_i}1_{\{\ell_i\cap K\neq \phi\}}$ and 
$m^K=\sum\limits_{i\in\mathbb{N}}c_i\delta_{\ell_i}1_{\{\ell_i\subset K\}}$. 
We denote by $\mathcal{F}_K$ the $\sigma$-field generated by $\{m_K:m\in M_{p}^{+}\}$ and 
by $\mathcal{F}^K$ the $\sigma$-field generated by $\{m^K:m\in M_{p}^{+}\}$. 
A random set $\mathcal{K}$ is called $(\mathcal{F}_{K})_{\{K\text{ finite}\}}$-optional iff for any deterministic $K\subseteq V$, $\{\mathcal{K}\subset K\}$ is $\mathcal{F}_{K}$-measurable. 
Then, define $$\mathcal{F}_{\mathcal{K}}=\{A\in\mathcal{F}:A\cap\{\mathcal{K}\subset K\}\in\mathcal{F}_K\}.$$
Similar definitions hold for the filtration $(\mathcal{F}^{K})_{\{K\text{ finite}\}}$.

\medskip

For $x\in\mathbb{Z}^d$ and natural number $n\in\mathbb{N}$, denote by $B(x,n)$ the box of side length $2n$ centered at $x$.

\subsection{Facts about random walk}
\begin{lem}[{\cite[Proposition 4.6.4]{LawlerMR2677157}}]\label{lem:hitting probability}
 For a transient graph and a subset $F$ of vertices, by last passage time decomposition,
$$\mathbb{P}^x[\tau(F)<\infty]=\sum\limits_{z\in\partial F}G(x,z)\mathbb{P}^z[\tau^{+}(F)=\infty].$$
\end{lem}

\begin{lem}[{\cite[Theorem 4.3.1]{LawlerMR2677157}}]\label{lem:Green function}
 For simple random walk on $\mathbb{Z}^d$, $d\geq 3$, there exist $0<c(d)\leq C(d)<\infty$ such that
 $$c(d)(1+||x-y||_{\infty})^{2-d}\leq G(x,y)\leq C(d)(1+||x-y||_{\infty})^{2-d}.$$
 More precisely, $G(x,y)=\frac{d\Gamma(d/2)}{(d-2)\pi^{d/2}}(||x-y||_{2}+1)^{2-d}+O((||x-y||_2+1)^{-d})$.
\end{lem}

\begin{lem}[{\cite[Proposition 6.5.1]{LawlerMR2677157}}]\label{lem: capacity}
 There exist $0<c(d)<C(d)<\infty$ such that for $n\geq 1$,
 $$c(d)n^{d-2}\leq\Capa(B(0,n))\leq C(d)n^{d-2}.$$
\end{lem}

The following lemma provides an estimate on the capacity of the random walk range.
\begin{lem}\label{l:lower bound SRW capacity in box}
For a SRW on $\mathbb Z^d$, $d\geq 3$, there exists $c(d)>0$ such that
\[
\inf\limits_{n\geq 1,z\in\partial B(0,n)}\mathbb{P}^{0}\left[\Capa(\{X_0,\ldots,X_{\tau(\partial B(0,n))}\})>c(d)\cdot F(d,n)|X_{\tau(\partial B(0,n))}=z\right]>0,
\]
where $F(d,n)=1_{d=3}\cdot n+1_{d=4}\cdot \frac{n^2}{\log n}+1_{d\geq 5}\cdot n^2$.
\end{lem}
\begin{proof}
It suffices to show that there exists $c' =c'(d)>0$ such that for all $T\geq 0$, 
\begin{equation}\label{eq:clbScib1}
\mathbb{P}^0\left[\Capa(\{X_0,\ldots,X_{T}\})\geq c'F(d,\sqrt{T})\right]>c'.
\end{equation}
Indeed, let $\tau(n) = \tau(\partial B(0,n))$. 
By the strong Markov property and Harnack's inequality,
\begin{multline*}
\mathbb{P}^{0}[\Capa(\{X_0,\ldots,X_{\tau(n)}\})>cF(d,n)|X_{\tau(n)}=z]\\
\geq \frac{\inf\limits_{w\in\partial B(0,\lceil n/2\rceil)}\mathbb{P}^w[X_{\tau(n)}=z]}{\mathbb{P}^{0}[X_{\tau(n)}=z]}\cdot \mathbb{P}^{0}[\Capa(\{X_0,\ldots,X_{\tau(\lceil n/2\rceil )}\})>cF(d,n)]\\
 \geq c''(d)\cdot \mathbb{P}^{0}[\Capa(\{X_0,\ldots,X_{\tau(\lceil n/2\rceil )}\})>cF(d,n)].
\end{multline*}
By Kolmogorov's maximal inequality for the coordinates,
\[
\mathbb{P}^0[\tau(\lceil n/2\rceil)<\delta n^2]\leq 4\delta.
\]
We choose $\delta = \frac{c'}{8}$ and apply \eqref{eq:clbScib1} with $T = \lfloor \frac {c'}{8} n^2\rfloor$ to 
get $\mathbb{P}^{0}[\Capa(\{X_0,\ldots,X_{\tau(\lceil n/2\rceil )}\})>cF(d,n)]\geq \frac {c'}{2}$ 
for a suitable choice of $c=c(d)$. 

\medskip

It remains to verify \eqref{eq:clbScib1}. 
By the Paley-Zygmund inequality it suffices to check that for some $0<c(d)\leq C(d)<\infty$, 
\[
\mathbb E^0[\Capa(\{X_0,\ldots,X_{T}\})]\geq c F(d,\sqrt{T})
\]
and
\[
\mathbb E^0[\Capa(\{X_0,\ldots,X_{T}\})^2]\leq C \left(\mathbb E^0[\Capa(\{X_0,\ldots,X_{T}\})]\right)^2.
\]
The first inequality was proved in \cite[Lemma 4]{ArtemMR2819660} for all $d\geq 3$, and 
the second was proved in \cite{ArtemMR2819660} for $d=3$ and $d\geq 5$, see the proof of Lemma~5 there. 
For $d=4$, the second inequality is obtained in \cite{ArtemMR2819660} with a logarithmic correction, 
which is not enough to imply \eqref{eq:clbScib1}. 
Below we provide a proof of the correct bound using a result about intersection of SRWs from \cite[Theorem 2.2]{LawlerMR679202}. 
We prove that there exists $C$ such that for the SRW on $\mathbb Z^4$, 
\begin{equation}\label{eq:clbScib2}
\mathbb{E}^0\left[\Capa(\{X_0,\ldots,X_{T}\})^2\right]\leq C\cdot \frac{T^2}{(\log T)^2}.
\end{equation}
Let $(X^0_n)_{n\geq 0},(X^1_n)_{n\geq 0},(X^2_n)_{n\geq 0}$ be three independent SRWs.
Denote by $\mathbb{E}_{(i)}^{x}$ the expectation corresponding to the random walk $X^i$ with initial point $x$.
Similarly, we define $(\mathbb{E}_{(i),(j)}^{x,y})_{i\neq j}$.
For simplicity of notation, we denote by $\mathbb{E}^{x,y,z}$ (or $\mathbb{P}^{x,y,z}$) the expectation (or probability) corresponding to $X^0,X^1$ and $X^2$ with initial points $x,y$ and $z$, respectively.
Denote by $X^0[0,T]$ the range of $X^0$ up to time $T$. Similarly, we define $X^1[0,\infty[$ and $X^2[0,\infty[$.

Let $x_0 = (2T,0,\dots,0)$. By Lemmas \ref{lem:hitting probability} and \ref{lem:Green function},
$\mathbb{E}^0[\Capa(\{X_0,\ldots,X_{T}\})^2]$ is comparable to
\[
T^{4}\cdot \mathbb{P}^{0,x_0,x_0}\left[X^0[0,T]\cap X^1[0,\infty[\neq\phi,~X^0[0,T]\cap X^2[0,\infty[\neq\phi\right].
\]
For $i=1$ and $2$, define $\tau_i=\inf\{j\geq 0:X^0_j\in X^{i}[0,\infty[\}$. By symmetry,
\[
\mathbb{P}^{0,x_0,x_0}\left[X^0[0,T]\cap X^1[0,\infty[\neq\phi,~X^0[0,T]\cap X^2[0,\infty[\neq\phi\right]\leq 2\mathbb{P}^{0,x_0,x_0}[\tau_1\leq\tau_2\leq T].
\]
By conditioning on $X^1$ and $X^2$ and then applying the strong Markov property for $X^0$ at time $\tau_1$,
\begin{align*}
 \mathbb{P}^{0,x_0,x_0}[\tau_1\leq\tau_2\leq T]=&\mathbb{E}^{0,x_0,x_0}\left[\tau_1\leq\tau_2,\tau_1<T,\mathbb{E}_{(0)}^{X^0_{\tau_1}}\left[1_{\{X^0[0,T-\tau_1]\cap X^2[0,\infty[\neq\phi]\}}\right]\right]\\
 \leq&\mathbb{E}^{0,x_0,x_0}\left[\tau_1<T,\mathbb{E}_{(0)}^{X^0_{\tau_1}}\left[1_{\{X^0[0,T]\cap X^2[0,\infty[\neq\phi\}}\right]\right].
\end{align*}
Then, we take the expectation with respect to $X^2$ and get that
\begin{align*}
 \mathbb{P}^{0,x_0,x_0}[\tau_1\leq\tau_2\leq T]\leq &\mathbb{P}_{(0),(1)}^{0,x_0}[X^0[0,T]\cap X^1[0,\infty[\neq\phi]\\
 &\times\sup\limits_{y\in B(0,T)}\mathbb{P}_{(0),(2)}^{y,x_0}[X^0[0,T]\cap X^2[0,\infty[\neq\phi].
\end{align*}
Note that $||y-x_0||_{2}\geq T$ for all $y\in B(0,T)$. 
By \cite[Theorem 2.2]{LawlerMR679202}, there exists $C<\infty$ such that
\[
\sup\limits_{y\in B(0,T)}\mathbb{P}_{(0),(i)}^{y,x_0}[X^0[0,T]\cap X^i[0,\infty[\neq\phi]\leq C\frac{1}{T\log T}.
\]
Thus, \eqref{eq:clbScib2} is proved and the proof of Lemma~\ref{l:lower bound SRW capacity in box} is complete.
\end{proof}

\subsection{Properties of the loop measure \texorpdfstring{$\mu$}{µ}}
In this subsection, we present several properties of loop measure $\mu$. Most of them are taken from \cite{loop}, \cite{LejanMR2971372} and \cite{LemaireLeJan}.
\begin{lem}[Proposition 18 in Chapter 4 of \cite{loop}]\label{lem:non-trivial loop visiting F}
For a finite subset of vertices $F$ of a transient graph, 
\[
\mu(\ell~:~\ell\cap F \neq \phi)=\log\det(G|_{F\times F}),
\]
where $G$ is the Green function viewed as a matrix $\left(G(x,y)\right)_{x,y}$, and 
$G|_{F\times F}$ is its sub-matrix with indexes on $F\times F$. 

As a corollary, for $n$ different vertices $x_1,\dots,x_n$,
\[
\mu(\ell~:~\text{$x_i\in \ell$ for $i\in\{1,\dots,n\}$})=\sum\limits_{A\subset\{x_1,\ldots,x_n\},A\neq\phi}(-1)^{\# A+1}\log\det(G|_{A\times A}).
\]
\end{lem}
Our definition of $G$ is slightly different from that in \cite{loop}. Thus, we provide a proof.
\begin{proof}
We only prove the first part as the second part follows from the inclusion-exclusion principle.

Take an increasing sequence of finite sets $(B_n)_n$ which exhausts our graph. Moreover, we can take $B_1=F$. Then,
\begin{align*}
\mu(\ell\cap F\neq\phi)&=\lim\limits_{n\rightarrow\infty}\mu(\ell\cap F\neq\phi,~\ell\subset B_n)\\
&=\lim\limits_{n\rightarrow\infty}\left\{\mu(\ell\subset B_n)-\mu(\ell\subset B_n\setminus F)\right\}\\
&=\lim\limits_{n\rightarrow\infty}\left\{\sum\limits_{k\geq 2}\frac{\Tr (Q|_{B_n\times B_n})^k}{k}-\sum\limits_{k\geq 2}\frac{\Tr (Q|_{(B_n\setminus F)^2})^k}{k}\right\}\\
&=\lim\limits_{n\rightarrow\infty}\left\{\log\det(I-Q|_{(B_n\setminus F)^2})-\log\det(I-Q|_{B_n\times B_n})\right\}\\
&=\lim\limits_{n\rightarrow\infty}\log\frac{\det(I-Q|_{(B_n\setminus F)^2})}{\det(I-Q|_{B_n\times B_n})}.
\end{align*}
By Jacobi's equality,
\[
\frac{\det(I-Q|_{(B_n\setminus F)^2})}{\det(I-Q|_{B_n\times B_n})}=
\det((I-Q|_{B_n\times B_n})^{-1}|_{F\times F}).
\]
Since $\lim\limits_{n\rightarrow\infty}(I-Q|_{B_n\times B_n})^{-1}|_{F\times F}=G|_{F\times F}$, the result follows.
\end{proof}

The following lemma is a special case of the result in \cite[(4.3)]{loop}. In fact, Le Jan proves that the joint distribution of visiting times for a set of points is multi-variate binomial distribution. The result about the excursions can be derived from explicit calculation.
\begin{lem}[\cite{loop}]\label{lem:cut the loop passing through a vertex into excursions}
Fix a vertex $x_0$ in a transient graph with the Green function $G$. 
Let $\xi(x_0,\ell)$ count the number of visits of vertex $x_0$ in the loop $\ell$. 
Set $\xi(x_0,\mathcal{L}_{\alpha})=\sum\limits_{\ell\in\mathcal{L}_{\alpha}}\xi(x_0,\ell)$ be the total number of visits of $x_0$ for the loop ensemble $\mathcal{L}_{\alpha}$. 
Then, $\xi(x_0,\mathcal{L}_{\alpha})$ follows a negative binomial (or P\'{o}lya) distribution, i.e., 
\[
\mathbb{P}[\xi(x_0,\mathcal{L}_{\alpha})=k]=G(x_0,x_0)^{-\alpha}\cdot \left(1-\frac{1}{G(x_0,x_0)}\right)^k\cdot \frac{\alpha(\alpha+1)\cdots (\alpha+k-1)}{k!}.
\]
By cutting down all the loops from $\mathcal L_{\alpha}$ into excursions from $x_0$, we get $\xi(x_0,\mathcal{L}_{\alpha})$-many excursions.
Conditionally on $\xi(x_0,\mathcal{L}_{\alpha})=k$, those excursions are i.i.d. sample of the SRW excursions with finite length.
\end{lem}

\bigskip

We proceed by describing a useful representation of the measure of a given loop visiting two disjoint sets 
as a linear combination of the measures of based loops starting on the boundary of one of the sets,
see \eqref{eq:muS1S2}. 
We first introduce some notation. 
For a based loop $\dot\ell$, the multiplicity of $\dot\ell$ is defined as 
\[
m(\dot\ell) = \max\{k\geq 1~:~\dot\ell = (\dot\ell_1,\dots,\dot\ell_k)~\text{for $\dot\ell_1 = \dots =\dot\ell_k$}\}.
\]
The multiplicity of a loop $\ell$, denoted also by $m(\ell)$, is the multiplicity of any of the based loops in the equivalence class. 
By \eqref{def:dotmu} and the definition of $\mu$ as the push-forward of $\dot\mu$, 
for any loop $\ell$ of length $n$, 
\begin{equation}\label{eq:mumultiplicity}
\mu_\kappa(\ell) = \frac{n}{m(\ell)}\cdot \dot\mu_\kappa(\dot\ell).
\end{equation}

For two disjoint subsets $S_1$ and $S_2$, consider the map $L(S_1,S_2)$ from the space of loops visiting $S_1$ and $S_2$ to subsets of based loops such that 
\begin{itemize}\itemsep0pt
\item[(a)]
any $\dot\ell = (x_1,\dots,x_n)\in L(S_1,S_2)(\ell)$ is in the equivalence class $\ell$, 
\item[(b)]
$x_1\in S_1$, 
\item[(c)]
there exists $i$ such that $x_i\in S_2$ and $x_j\notin (S_1\cup S_2)$ for all $j>i$. 
\end{itemize}
Note that $L(S_1,S_2)(\ell)\neq\phi$ if and only if $\ell$ visits $S_1$ and $S_2$. 

For any $\dot\ell = (x_1,\dots,x_n)\in L(S_1,S_2)(\ell)$, we define recursively the sequence $(\tau_i)_{i\geq 0}$ as follows: 
for $k\geq 0$,
\[
\tau_0 = 1,\quad \tau_{2k+1} = \inf\{j>\tau_{2k}~:~ x_j\in S_2\},\quad \tau_{2k+2} = \inf\{j>\tau_{2k+1}~:~x_j\in S_1\}.
\]
We write $\inf\{\phi\} = n+1$. 
By the definition of $L(S_1,S_2)$, there exists $k(\dot\ell)\geq 1$ such that $\tau_{2k(\dot\ell)-1} \leq n$ and $\tau_{2k(\dot\ell)} = n+1$. 
The value of $k(\dot\ell)$ is the same for all $\dot\ell\in L(S_1,S_2)(\ell)$, and we denote it by $k(\ell)$. 
\begin{claim}
For any loop $\ell$ of length $n$ visiting $S_1$ and $S_2$,  
\begin{equation}\label{eq:muS1S2}
\mu_\kappa(\ell) = \frac{n}{k(\ell)}\cdot \sum_{\dot\ell\in L(S_1,S_2)(\ell)}\dot\mu_\kappa(\dot\ell) .\
\end{equation}
\end{claim}
\begin{proof}
The claim is immediate from the fact that $k(\ell) = m(\ell)\cdot |L(S_1,S_2)(\ell)|$ and \eqref{eq:mumultiplicity}.
\end{proof}

\bigskip

We end this section with crucial estimates which will be frequently used in the proofs.
\begin{lem}\label{lem:one loop connection on Zd}\ 
\begin{itemize}\itemsep0pt
\item[a)] 
For $d\geq 3$ and $\lambda>1$, there exists $C = C(d,\lambda)<\infty$ such that for $N\geq 1$, $M\geq\lambda N$, and $K\subset B(0,N)$,
\[
\mu(\ell~:~K\stackrel{\ell}\longleftrightarrow \partial B(0,M))\leq C\cdot \Capa(K)\cdot M^{2-d}.
\]
\item[b)] 
For $d\geq 3$, let $F(d,n)=1_{d=3}\cdot n+1_{d=4}\cdot\frac{n^2}{\log n}+1_{d\geq5}\cdot n^2$. 
For any $\lambda>1$, there exists $c = c(d,\lambda)>0$ such that for $N\geq 1$, $M\geq \lambda N$, and $K\subset B(0,N)$, 
\[
\mu\left(\ell~:~K\stackrel{\ell}\longleftrightarrow \partial B(0,M),~ \Capa(\ell)>cF(d,M)\right)\geq c\cdot\Capa(K)\cdot M^{2-d}.
\]
\end{itemize}
\end{lem}
\begin{proof}\
\begin{itemize}
\item[a)] 
Let $(\tau_n)_{n\geq 0}$ be the sequence of stopping times defined recursively by
\begin{align*}
\tau_0\overset{\text{def}}{=}&\tau(K),\\
\tau_{2k+1}\overset{\text{def}}{=}&\inf\{n>\tau_{2k}:X_n\in \partial B(0,M)\},\\
\tau_{2k+2}\overset{\text{def}}{=}&\inf\{n>\tau_{2k+1}:X_n\in K\}.
\end{align*}
By \eqref{def:dotmu} and \eqref{eq:muS1S2}, 
\[
\mu(\ell~:~K\stackrel{\ell}\longleftrightarrow \partial B(0,M))=\sum\limits_{n\geq 1}\frac 1n \sum\limits_{x\in \partial K}\mathbb{P}^x[X_{\tau_{2n}}=x].
\]
By the strong Markov property, for $n\geq 0$,
\begin{equation*}
\mathbb{P}^x[X_{\tau_{2n+2}}=x]=\sum\limits_{\substack{y\in \partial K\\z\in\partial B(0,M)}}\mathbb{P}^x[X_{\tau_{2n}}=y]\mathbb{P}^y[X_{\tau(\partial B(0,M))}=z]\mathbb{P}^z[X_{\tau(K)}=x].
\end{equation*}
By Harnack's inequality, there exists a constant $C = C(d,\lambda)$ such that
\[
\max\limits_{y\in \partial K}\mathbb{P}^y[X_{\tau(\partial B(0,M))}=z]\leq C\cdot\mathbb{P}^0[X_{\tau(\partial B(0,M))}=z].
\]
Therefore,
\begin{align*}
\mathbb{P}^x[X_{\tau_{2n+2}}=x]\leq &C\cdot\mathbb{P}^x[\tau_{2n}<\infty]\sum\limits_{z\in\partial B(0,M)}\mathbb{P}^0[X_{\tau(\partial B(0,M))}=z]\mathbb{P}^z[X_{\tau(K)}=x]\\
\leq &C\cdot(\max\limits_{x\in\partial K}\mathbb{P}^x[\tau_{2n}<\infty])\cdot \mathbb{P}^0[X_{\tau_2}=x]\\
\leq &C\cdot (\max\limits_{x\in\partial K}\mathbb{P}^x[\tau_{2}<\infty])^n\cdot\mathbb{P}^0[X_{\tau_2}=x].
\end{align*}
Under the assumption $\lambda>1$, there exists $\rho = \rho(d,\lambda)<1$ such that 
\[
\max\limits_{x\in \partial K}\mathbb{P}^x[\tau_2<\infty]<\rho<1.
\]
Thus, there exists $C' = C'(d,\lambda)$ such that 
\begin{multline*}
\mu(\ell~:~K\stackrel{\ell}\longleftrightarrow \partial B(0,M))
\leq 
C\cdot \mathbb{P}^0[\tau_2<\infty]\cdot\left(\sum\limits_{n\geq 1}\frac{1}{n}\left(\max\limits_{x\in\partial K}\mathbb{P}^x[\tau_{2}<\infty]\right)^{n-1}\right)\\
\leq C'\cdot \mathbb{P}^0[\tau_2<\infty]
\leq C'\cdot \max\limits_{z\in\partial B(0,M)}\mathbb{P}^{z}[\tau(K)<\infty].
\end{multline*}
By Lemmas~\ref{lem:hitting probability} and \ref{lem:Green function} and using the assumption $\lambda>1$, there exists $C'' = C''(d,\lambda)$ such that for any $z\in\partial B(0,M)$,
\[
\mathbb{P}^z[\tau(K)<\infty]\leq C''\cdot M^{2-d}\cdot\sum\limits_{w\in\partial K}\mathbb{P}^{w}[\tau^{+}(K)=\infty]
= C''\cdot M^{2-d}\cdot\Capa(K).
\]
The result follows.
\item[b)] By Lemma \ref{l:lower bound SRW capacity in box} and the monotonicity of the capacity of finite sets, there exists $c =c(d,\lambda)>0$ such that
\begin{multline*}
\min\limits_{y\in\partial B(0,M)}\mathbb{P}^y[\Capa(\{X_0,\ldots,X_{\tau(K)}\})>cF(d,M)|X_{\tau(K)}=x]\notag\\
\geq \min\limits_{y\in\partial B(0,M)}\mathbb{P}^y[\Capa(\{X_0,\ldots,X_{\tau(\partial B(y,M-N))}\})>cF(d,M)|X_{\tau(K)}=x]\notag\\
\geq \min\limits_{z\in\partial B(0,M-N)}\mathbb{P}^0[\Capa(\{X_0,\ldots,X_{\tau(\partial B(0,M-N))}\})>cF(d,M)|X_{\tau(\partial B(0,M-N))}=z]\notag\\
\geq c.
\end{multline*}
By \eqref{def:dotmu} and \eqref{eq:muS1S2} (and ignoring the loops with $k(\ell) \geq 2$), 
\begin{multline*}
\mu\left(\ell~:~K\stackrel{\ell}\longleftrightarrow \partial B(0,M),~ \Capa(\ell)>cF(d,M)\right)\geq\\
\sum\limits_{x\in \partial K,y\in \partial B(0,M)}\mathbb{P}^x[X_{\tau(\partial B(0,M))}=y]\cdot \mathbb{P}^y[X_{\tau(K)}=x,\Capa(\{X_0,\ldots,X_{\tau(K)}\})>cF(d,M)]\\
\geq c\cdot\sum\limits_{x\in \partial K,y\in \partial B(0,M)}\mathbb{P}^x[X_{\tau(\partial B(0,M))}=y]\cdot\mathbb{P}^y[X_{\tau(K)}=x].
\end{multline*}
The rest of the proof is very similar to that of Part a). 
It is based on an application of Harnack's inequality, Lemmas~\ref{lem:hitting probability} and \ref{lem:Green function}.
We omit the details. \qedhere
\end{itemize}
\end{proof}

\section{Some basic properties of loop percolation}\label{sec: basic properties of loop percolation}

In this section we collect some elementary properties of the loop percolation on $\mathbb Z^d$:
\begin{itemize}\itemsep0pt
 \item long range correlations, see Proposition \ref{prop: long range correlation},
 \item translation invariance and ergodicity, see Proposition \ref{prop: translation invariance and ergodicity},
 \item the uniqueness of infinite cluster, see Proposition \ref{prop: uniqueness of infinite cluster},
 \item the existence of percolation for $\alpha>0$ and $\kappa<0$, see Proposition \ref{prop: existence of percolation for negative kappa on Zd}.
\end{itemize}
We remark that except for the translation invariance, these properties will not be used in the proofs of the main results.  

\bigskip

The following proposition shows the long range correlations in SRW loop percolation on $\mathbb{Z}^d$ for $d\geq 3$ and $\kappa=0$.
\begin{prop}\label{prop: long range correlation}
Let $d\geq 3$, $\alpha>0$, and $\kappa = 0$. For $x\in\mathbb Z^d$, we write $x\in \mathcal L_\alpha$ if there exists $\ell\in\mathcal L_\alpha$ such that $x\in\ell$. 
For any $x,y\in\mathbb Z^d$, 
\begin{multline*}
\Cov(1_{\{x\in\mathcal{L}_{\alpha}\}},1_{\{y\in\mathcal{L}_{\alpha}\}})
  =(G(x,x)G(y,y))^{-\alpha}\left(\left(1-\frac{G(x,y)G(y,x)}{G(x,x)G(y,y)}\right)^{-\alpha}-1\right)\\
\sim\alpha(G(0,0))^{-2-2\alpha}\frac{d^2(\Gamma(d/2))^2}{\pi^d(d-2)^2}||x-y||_2^{4-2d},\quad\text{ as }||x-y||_2\rightarrow\infty.
 \end{multline*}
\end{prop}
\begin{proof}
By Lemma \ref{lem:non-trivial loop visiting F}, for $F\subset\mathbb Z^d$, 
\[
\mathbb{P}[\forall \ell\in\mathcal{L}_{\alpha}~:~\ell\cap F=\phi]=\exp\{-\alpha\mu(\ell~:~\ell\cap F\neq\phi)\}
=(\det(G|_{F\times F}))^{-\alpha}.
\]
Thus,
\begin{multline*}
\Cov(1_{\{x\in\mathcal{L}_{\alpha}\}},1_{\{y\in\mathcal{L}_{\alpha}\}})=
\mathbb P[x,y\notin\mathcal L_\alpha] - \mathbb P[x\notin\mathcal L_\alpha]\cdot\mathbb P[y\notin\mathcal L_\alpha]\\
=\begin{vmatrix}
G(x,x) & G(x,y)\\
G(y,x) & G(y,y)
\end{vmatrix}^{-\alpha}
- (G(x,x)G(y,y))^{-\alpha}.
\end{multline*}
This gives the first statement. The second follows from Lemma~\ref{lem:Green function}. 
\end{proof}

\medskip

Next, we prove the ergodicity of SRW loop soup under lattice shifts $(t_x)_{x\in\mathbb{Z}^d}$, $t_x:\ell\mapsto \ell+x$.
\begin{prop}\label{prop: translation invariance and ergodicity}
The Poisson loop ensemble associated with a simple random walk on $\mathbb{Z}^d$ is invariant under lattice shifts. Moreover, it is ergodic under these translations.
\end{prop}

\begin{proof}[Proof of Proposition \ref{prop: translation invariance and ergodicity}]
The translation invariance of Poisson loop ensemble comes from the translation invariance of its intensity measure and we omit its proof here. For the ergodicity, let us fix $x\in\mathbb{Z}^d$, a measurable event $A$ and show that
$$\lim\limits_{n\rightarrow\infty}\frac{1}{n}\sum\limits_{i=1}^{n}1_{A}\circ t_x^i=\text{constant,}\quad\mathbb{P}\text{-almost surely}.$$
By the law of large numbers, it is enough to show that for measurable events $A$ and $B$,
$$\lim\limits_{n\rightarrow\infty}\Cov(1_B,1_A\circ t_x^n)=0.$$
Or equivalently, by translation invariance, it is enough to show
$$\lim\limits_{n\rightarrow\infty}\mathbb{E}[1_{B}\cdot 1_{A}\circ t_x^n]=\mathbb{P}[A]\mathbb{P}[B].$$
By a classical monotone class argument (Dynkin's $\pi-\lambda$ theorem), it is enough to show it for $A,B\in \bigcup\limits_{K\text{ finite}}\mathcal{F}^{K}$. (Recall that $\mathcal{F}^K$ is the sigma-filed generated by loops inside $K$.) We choose $K$ large enough such that $A$ and $B$ are $\mathcal{F}^{K}$-measurable. Then, $1_A\circ t_x^n$ is $\mathcal{F}^{t_x^{-n}(K)}$-measurable. For $n$ large enough, $t_x^{-n}(K)\cap K$ is empty. Recall that Poisson random measure of disjoint sets are independent. Then, for the same $n$, we have independence between $1_B$ and $1_A\circ t_x^n$ since we have independence between the loops inside $t_x^{-n}(K)$ and those loops inside $K$. Thus,
\begin{equation*}
 \lim\limits_{n\rightarrow\infty}\mathbb{E}[1_B\cdot 1_A\circ t_x^n]=\mathbb{P}[A]\mathbb{P}[B].\qedhere
\end{equation*}
\end{proof}

Ergodicity implies that any translation invariant event has probability either $0$ or $1$. 
Consequently, 
\begin{equation}\label{equivalence:equivalence between theta and p}
\theta(\alpha,\kappa)>0\Longleftrightarrow \mathbb P[\exists x\in\mathbb Z^d~:~\#\mathcal C_{\alpha,\kappa}(x)=\infty]=1.
\end{equation}

The next proposition states the uniqueness of infinite cluster.

\begin{prop}[Uniqueness of infinite cluster]\label{prop: uniqueness of infinite cluster}
 For the Poisson loop ensemble associated with a SRW on $\mathbb{Z}^d$, there is at most one infinite cluster in the corresponding loop percolation.
\end{prop}
\begin{proof}
By Theorem 1 in \cite{GandolfiMR1169017}, ``translation invariance'' and ``positive finite energy property'' imply the uniqueness. 
Thus, we only need to show the positive finite energy property: for all $e=\{x,y\}$, 
\[
\mathbb{P}\left[\omega_{e}=1~|~\sigma(\{\omega_f:f\text{ is an edge in }\mathbb{Z}^d\text{ and }f\neq e\})\right]>0\text{ almost surely},
\]
where $\omega_e\in \{0,1\}$ and $\omega_e =1$ if and only if the edge $e$ is traversed by a loop from $\mathcal L_\alpha$. 

From the independence structure of Poisson loop ensemble between two disjoint sets, 
the event $\{\text{the loop $(x,y)$ is in $\mathcal L_\alpha$}\}$ is independent from 
$\sigma(\{\omega_f:f\text{ is an edge in }\mathbb{Z}^d\text{ and }f\neq e\})$.
Thus,
\begin{multline*}
\mathbb{P}\left[\omega_{e}=1~|~\sigma(\{\omega_f:f\text{ is an edge in }\mathbb{Z}^d\text{ and }f\neq e\})\right]\\
\geq \mathbb{P}\left[\{\text{the loop $(x,y)$ is in $\mathcal L_\alpha$}\}~|~\sigma(\{\omega_f:f\text{ is an edge in }\mathbb{Z}^d\text{ and }f\neq e\})\right]\\
=\mathbb{P}\left[\text{the loop $(x,y)$ is in $\mathcal L_\alpha$}\right]>0. \text{\qedhere}
\end{multline*}
\end{proof}

\begin{rem}\label{rem:finite energy}
In fact, for $d\geq 3$, for an edge $e=\{x,y\}$ in the integer lattice $\mathbb{Z}^d$, one can also prove that
\[
\mathbb{P}\left[\omega_{e}=0~|~\sigma(\{\omega_f:f\text{ is an edge in }\mathbb{Z}^d\text{ and }f\neq e\})\right]>0\text{ almost surely}.
\]
As a consequence, there exists at most one infinite cluster of closed edges. 
\end{rem}

\medskip

We complete this section with a statement about triviality of loop percolation on $\mathbb Z^d$ for $\alpha>0$ and $\kappa<0$. 
In particular, it implies the second statement of Theorem~\ref{thm:phase transition}.

\begin{prop}\label{prop: existence of percolation for negative kappa on Zd}
 For $d\geq 3$, $\alpha>0$, $\kappa<0$, and $x\in\mathbb Z^d$, the graph $\mathbb{Z}^d$ is covered by the loops $\{\ell\in\mathcal{L}_{\alpha}:x\in \ell\}$ passing through $x$.
\end{prop}
\begin{proof}
For $y\neq x$, by \eqref{def:dotmu} and \eqref{eq:muS1S2} applied to $S_1 = \{x\}$ and $S_2 = \{y\}$, and by ignoring all the loops with $k(\ell)\geq 2$, 
we get that for any $n\geq 1$, 
\begin{multline*}
\mu(\ell:x,y\in\ell) \geq 
\mathbb E^x \left[(1 + \kappa)^{-\tau(y)}\cdot 1_{\tau(y)<\infty}\right]\cdot 
\mathbb E^y \left[(1 + \kappa)^{-\tau(x)}\cdot 1_{\tau(x)<\infty}\right]\\
= \mathbb E^x \left[(1 + \kappa)^{-\tau(y)}\cdot 1_{\tau(y)<\infty}\right]^2
\geq 
(1+\kappa)^{-2n}\cdot \mathbb P^x\left[n\leq \tau(y)<\infty\right]^2.
\end{multline*}
Moreover, for any $n\geq 1$, 
\begin{multline*}
\mathbb P^x\left[n\leq \tau(y)<\infty\right] = 
\mathbb E^x\left[\tau(y)\geq n,\mathbb{P}^{X_n}[\tau(y)<\infty]\right]\\
\stackrel{\mathrm{Lemmas}~\ref{lem:hitting probability}, \ref{lem:Green function}}\geq 
c(d)\cdot \mathbb E^x\left[\tau(y)\geq n,(||X_n - y|| + 1)^{2-d}\right]\\
\geq c(d)\cdot \mathbb P^x[\tau(y) = \infty]\cdot (||x-y|| + n)^{2-d}
\geq c'(d)\cdot (||x-y|| + n)^{2-d} .
\end{multline*}
Since the above inequalities hold for all $n\geq 1$, $\mu(\ell:x,y\in\ell) = \infty$. 
Therefore, 
\[
\mathbb P[\exists \ell\in\mathcal L_\alpha~:~x,y\in\ell] = 1 - e^{-\alpha\cdot \mu(\ell~:~x,y\in\ell)} = 1.\qedhere
\]
\end{proof}

\section{First results for the one-arm connectivity for \texorpdfstring{$\mathbb{Z}^d$}{Zd} \texorpdfstring{$(d\geq 3)$}{(d≥3)} and \texorpdfstring{$\kappa=0$}{κ=0}}\label{sec: first results of one-arm connectivity}
\subsection{Lower bound: proof of Theorem \ref{thm:at most polynomial decay of the connectivity}}

For any $d\geq 3$, $\alpha>0$, and $n\geq 1$, 
\begin{multline*}
\mathbb{P}[0\overset{\mathcal{L}_{\alpha}}{\longleftrightarrow}\partial B(0,n)]
\geq \mathbb{P}[\exists \ell\in\mathcal{L}_{\alpha}~:~0\stackrel{\ell}\longleftrightarrow\partial B(0,n)]\\
=1-\exp\left\{-\alpha\mu(\ell~:~0\stackrel{\ell}\longleftrightarrow\partial B(0,n)) \right\}
\stackrel{\mathrm{Lemma}~\ref{lem:one loop connection on Zd}}\geq\alpha\cdot c(d)\Capa(\{0\})n^{2-d}.
\end{multline*}
\qed

\subsection{Upper bound: proof of Theorem~\ref{thm: polynomial decay}}

\subsubsection{\texorpdfstring{$\thran>0$}{α1>0}}

We will prove the following lemma. 
\begin{lem}\label{lem:uniform upper bound}
For $d\geq 3$ and $\beta>1$, there exists $C(d,\beta)<\infty$ such that for all $\alpha>0$ and $n\geq 1$,
\begin{equation}\label{eq:luub1}
\mathbb{P}[B(0,n)\overset{\mathcal{L}_{\alpha}}{\longleftrightarrow}\partial B(0,\lceil\beta n\rceil)]\leq C(d,\beta)\cdot\alpha.
\end{equation}
\end{lem}
Lemma~\ref{lem:uniform upper bound} implies that $\thran>0$ for all $d\geq 3$. 
\begin{proof}
Fix $d\geq 3$ and $\beta>1$. 
For $\alpha>0$ and $n\geq 1$, define the function
\[
f^{(\alpha)}(n)=\sup\limits_{k\leq n}\mathbb{P}[B(0,k)\overset{\mathcal{L}_{\alpha}}{\longleftrightarrow}\partial B(0,\lceil\beta k\rceil)].
\]
By Lemma~\ref{lem:one loop connection on Zd}, there exists $C_1 = C_1(d)<\infty$ such that
\[
f^{(\alpha)}(1)\leq \mathbb{P}[\exists \ell\in\mathcal{L}_{\alpha}~:~\ell\cap B(0,1)\neq\phi]\leq C_1\cdot \alpha. 
\]
We will prove that there exists $C_2 = C_2(d,\beta)<\infty$ such that for all $\alpha>0$ and $n\geq 1$, 
\begin{equation}\label{eq:renormalization}
f^{(\alpha)}(4n)\leq C_2\cdot \alpha + C_2\cdot (f^{(\alpha)}(n))^2.
\end{equation}
Before we prove \eqref{eq:renormalization}, we show how it implies \eqref{eq:luub1}. 
Take $C_3 = C_3(d,\beta)=\max\left(C_1,2C_2\right)$. 
On the one hand for $\alpha\leq C_3^{-2}$, by induction on $m$, we obtain that $f^{(\alpha)}(4^m)\leq C_3\cdot \alpha$ for all $m\geq 0$.
Then by monotonicity of $f^{(\alpha)}$, $f^{(\alpha)}(n)\leq C_3\cdot \alpha$ for all $\alpha\leq C_3^{-2}$ and $n\geq 1$.
On the other hand for $\alpha\geq C_3^{-2}$ and $n\geq 1$, $f^{(\alpha)}(n) \leq 1 \leq C_3^2\cdot \alpha$. 
Thus \eqref{eq:luub1} follows with $C(d,\beta) = \max(C_3,C_3^2)$. 

\medskip

It remains to prove \eqref{eq:renormalization}. For $x\in\mathbb Z^d$, $m\geq k\geq 1$, consider the events 
\[
E^\alpha(x,k,m) = \{B(x,k)\overset{\mathcal{L}_{\alpha}}{\longleftrightarrow}\partial B(x,m)\}.
\]

It suffices to show that 
\[
\mathbb P[E^\alpha(0,4k,\lceil 4\beta k\rceil)] \leq C(d,\beta)\cdot \alpha + C(d,\beta)\cdot \mathbb P[E^\alpha(0,k,\lceil \beta k\rceil)]^2.
\]
Moreover, we may suppose $k\geq \frac{1}{\beta-1}$. Let 
\[
a_k = 4k,\quad b_k = 4k + \lceil (\beta -1)k\rceil,\quad c_k = \lceil 4\beta k\rceil - \lceil (\beta-1) k\rceil, \quad d_k = \lceil 4\beta k\rceil.
\]
The key observation is that if $E^\alpha(0,a_k,d_k)$ occurs and 
$\mathcal L_\alpha$ does not contain a loop intersecting both $B(0,b_k)$ and $\partial B(0,c_k)$, 
then $B(x,a_k)$ is connected to $\partial B(0,b_k)$ by loops from $\mathcal L_\alpha$ which are contained in $B(0,c_k-1)$, and 
$B(0,c_k)$ is connected to $\partial B(0,d_k)$ by loops which are not contained in $B(0,c_k-1)$. 
Since the two collections of loops are disjoint and $\mathcal L_\alpha$ is a Poisson point process, the two events are independent. 
Thus, 
\begin{eqnarray*}
\mathbb P\left[E^\alpha(0,a_k,d_k)\right] &\leq 
&\mathbb P\left[\exists \ell\in\mathcal L_\alpha~:~B(0,b_k)\stackrel{\ell}\longleftrightarrow \partial B(0,c_k)\right]\\
&+ &\mathbb P\left[E^\alpha(0,a_k,b_k)\right]\cdot \mathbb P\left[E^\alpha(0,c_k,d_k)\right].
\end{eqnarray*}
Since $c_k>b_k$ and $\lim\limits_{k\rightarrow\infty}\frac{c_k}{b_k}=3$, by Lemmas~\ref{lem:one loop connection on Zd} and \ref{lem: capacity}, 
\[
\mathbb P\left[\exists \ell\in\mathcal L_\alpha~:~B(0,b_k)\stackrel{\ell}\longleftrightarrow\partial B(0,c_k)\right] \leq C(d,\beta)\cdot \alpha.
\]
It remains to show that for some $C(d,\beta)<\infty$, $\mathbb P\left[E^\alpha(0,a_k,b_k)\right]\leq C(d,\beta)\cdot \mathbb P[E^\alpha(0,k,\lceil\beta k\rceil)]$ 
and $\mathbb P\left[E^\alpha(0,c_k,d_k)\right]\leq C(d,\beta)\cdot \mathbb P[E^\alpha(0,k,\lceil\beta k\rceil)]$.

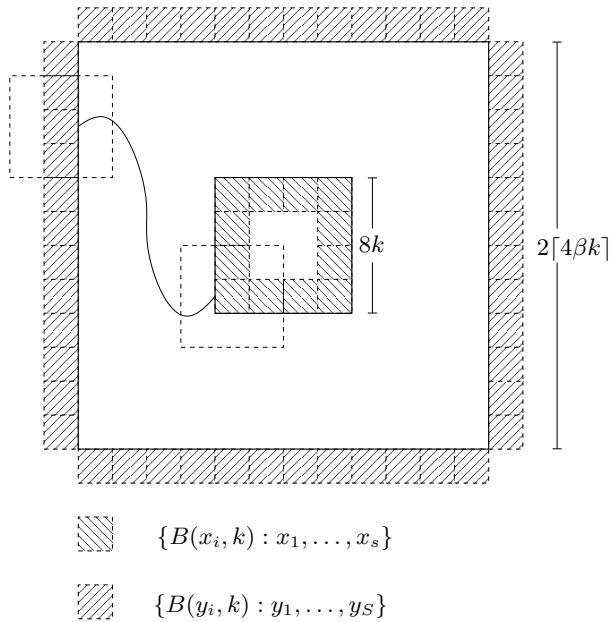
\begin{figure}[!htbp]
\begin{picture}(0,0)%
\includegraphics{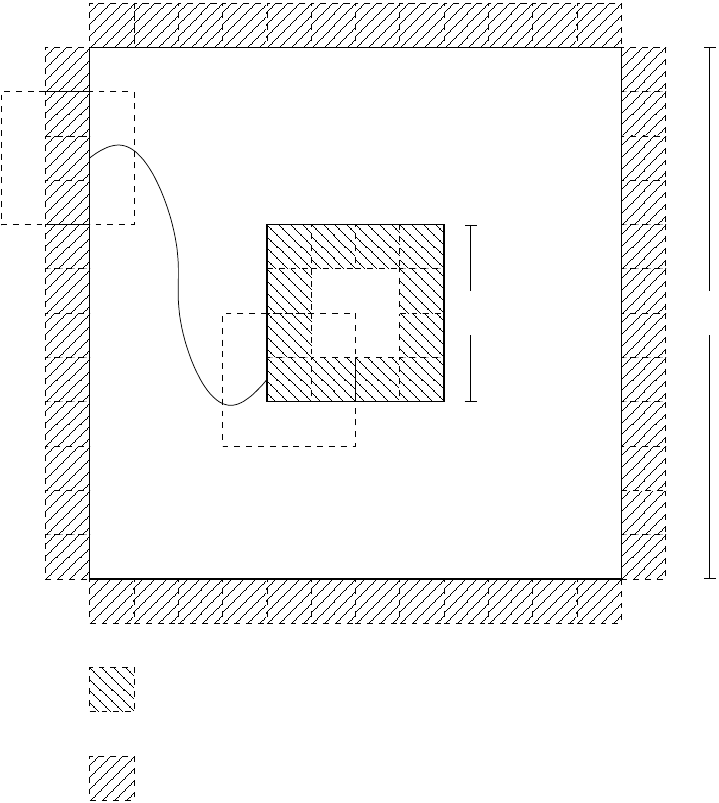}%
\end{picture}%
\setlength{\unitlength}{1865sp}%
\begingroup\makeatletter\ifx\SetFigFont\undefined%
\gdef\SetFigFont#1#2#3#4#5{%
  \reset@font\fontsize{#1}{#2pt}%
  \fontfamily{#3}\fontseries{#4}\fontshape{#5}%
  \selectfont}%
\fi\endgroup%
\begin{picture}(7284,8131)(889,-8630)
\put(2836,-7621){\makebox(0,0)[lb]{\smash{{\SetFigFont{9}{10.8}{\familydefault}{\mddefault}{\updefault}{\color[rgb]{0,0,0}$\{B(x_i,k):x_1,\ldots,x_s\}$}%
}}}}
\put(2791,-8521){\makebox(0,0)[lb]{\smash{{\SetFigFont{9}{10.8}{\familydefault}{\mddefault}{\updefault}{\color[rgb]{0,0,0}$\{B(y_i,k):y_1,\ldots,y_S\}$}%
}}}}
\put(5491,-3751){\makebox(0,0)[lb]{\smash{{\SetFigFont{9}{10.8}{\familydefault}{\mddefault}{\updefault}{\color[rgb]{0,0,0}$8k$}%
}}}}
\put(7831,-3796){\makebox(0,0)[lb]{\smash{{\SetFigFont{9}{10.8}{\familydefault}{\mddefault}{\updefault}{\color[rgb]{0,0,0}$2\lceil 4 \beta k\rceil$}%
}}}}
\end{picture}%
\caption{Illustration of event $E^{\alpha}(0,4k,\lceil 4\beta k\rceil)$ for $\beta=3$.}
\label{fig:renomalization}
\end{figure}

Let $s = s(d)$ be such that there exist $x_1,\dots,x_s\in \mathbb Z^d$ such that $B(x_i,k)\subseteq B(0,4k)$ and $\cup_{i=1}^s\partial B(x_i,k)\supseteq \partial B(0,4k)$.
Let $S = S(d,\beta)$ be such that there exist $y_1,\dots, y_S\in \mathbb Z^d$ such that $B(y_i,k)\cap B(0,\lceil 4\beta k\rceil-1)=\phi$ and 
$\cup_{i=1}^S\partial B(y_i,k)\supseteq \partial B(0,\lceil 4\beta k\rceil)$. 
Such a choice of $s$, $S$ always exists, and we fix some suitable $s$, $S$ and some corresponding $x_1,\dots, x_s$ and $y_1,\dots, y_S$. 
Note that $E^\alpha(0,a_k,b_k)$ implies that $E^\alpha(x_i,k,\lceil \beta k\rceil)$ occurs for some $1\leq i\leq s$, 
and $E^\alpha(0,c_k,d_k)$ implies that $E^\alpha(y_i,k,\lceil \beta k\rceil)$ occurs for some $1\leq i\leq S$. 
Thus using translation invariance, $\mathbb P[E^\alpha(0,a_k,b_k)]\leq s\cdot \mathbb P[E^\alpha(0,k,\lceil\beta k\rceil)]$ and 
$\mathbb P[E^\alpha(0,c_k,d_k)]\leq S\cdot \mathbb P[E^\alpha(0,k,\lceil\beta k\rceil)]$. 
The result follows. 
\end{proof}

\subsubsection{Upper bound on the one-arm probability for \texorpdfstring{$\alpha<\thran$}{α<α1}}

Let $d\geq 3$, $\alpha>0$ and $\beta>1$. 
Define two random sequences $(\mathcal{A}_n)_n$ and $(\mathcal{B}_n)_n$ as follows. Let $\mathcal B_n = \lceil \beta\mathcal A_n\rceil$ and 
\[
\mathcal A_1 = 1,\quad \mathcal A_n = \inf\left\{m: \begin{array}{c}\text{$B(0,m-1)$ contains all the loops from $\mathcal L_\alpha$}\\
\text{which intersect $B(0,\mathcal{B}_{n-1})$}\end{array}\right\}\quad (\text{for $n\geq 2$}).
\]
Since the graph $\mathbb{Z}^d$ is transient, the total mass under $\mu$ of the loops intersecting $B(0,N)$ is finite for any $N$ according to Lemma \ref{lem:non-trivial loop visiting F}. 
Therefore, in the Poisson loop ensemble $\mathcal{L}_{\alpha}$, the number of the loops intersecting $B(0,N)$ is almost surely finite. 
Thus, $\mathcal A_n$ are almost surely finite for all $n$. 

Consider
\begin{equation}\label{def:r}
r = r(d,\alpha,\beta) \stackrel{\mathrm{def}}= \sup_{n\geq 1}\mathbb P\left[B(0,n)\stackrel{\mathcal L_\alpha}\longleftrightarrow\partial B(0,\lceil \beta n\rceil)\right].
\end{equation}
We first show that for all $n\geq 1$, 
\begin{equation}\label{eq:one arm to mathcalBn}
\mathbb P\left[0\stackrel{\mathcal L_\alpha}\longleftrightarrow \partial B(0,\mathcal B_n)\right]\leq r^n .
\end{equation}
The proof is by induction on $n$. 
The case $n=1$ follows from the definitions of $r$, $\mathcal A_1$ and $\mathcal B_1$. 
For $n\geq 2$, 
\begin{multline*}
\mathbb P\left[0\stackrel{\mathcal L_\alpha}\longleftrightarrow \partial B(0,\mathcal B_n)\right]
\leq \mathbb P\left[0\stackrel{\mathcal L_\alpha}\longleftrightarrow \partial B(0,\mathcal B_{n-1}),~B(0,\mathcal A_n)\stackrel{\mathcal L_\alpha}\longleftrightarrow \partial B(0,\mathcal B_n)\right]\\
=\sum_{(b_{n-1},a_n)} \mathbb P\left[\mathcal B_{n-1} = b_{n-1},~\mathcal A_n = a_n,~
0\stackrel{\mathcal L_\alpha}\longleftrightarrow \partial B(0,b_{n-1}),~B(0,a_n)\stackrel{\mathcal L_\alpha}\longleftrightarrow \partial B(0,\lceil\beta a_n\rceil)\right].
\end{multline*}
Using the fact that the loops intersecting $B(0,\mathcal{B}_{n-1})$ never visit $\partial B(0,\mathcal{A}_n)$, 
we can rewrite summands in the above display as
\[
\mathbb P\left[\mathcal B_{n-1} = b_{n-1},~\mathcal A_n = a_n,~
0\stackrel{\mathcal L_\alpha}\longleftrightarrow \partial B(0,b_{n-1}),~B(0,a_n)\stackrel{\mathcal L^{B(0,b_{n-1})^c}_\alpha}\longleftrightarrow \partial B(0,\lceil\beta a_n\rceil)\right].
\]
The event $\{\mathcal B_{n-1} = b_{n-1},~\mathcal A_n = a_n,~0\stackrel{\mathcal L_\alpha}\longleftrightarrow \partial B(0,b_{n-1})\}$ 
is a measurable function of loops intersecting $B(0,b_{n-1})$, 
and the event $\{B(0,a_n)\stackrel{\mathcal L^{B(0,b_{n-1})^c}_\alpha}\longleftrightarrow \partial B(0,\lceil\beta a_n\rceil)\}$ 
depends only on loops which do not intersect $B(0,b_{n-1})$. 
Thus, the two events are independent. 
Moreover, the random loops $\mathcal L^{B(0,b_{n-1})^c}_\alpha$ avoiding $B(0,b_{n-1})$ is a subset of $\mathcal L_\alpha$. 
Thus, by monotonicity, 
\[
\mathbb P\left[B(0,a_n)\stackrel{\mathcal L^{B(0,b_{n-1})^c}_\alpha}\longleftrightarrow \partial B(0,\lceil\beta a_n\rceil)\right]
\leq \mathbb P\left[B(0,a_n)\stackrel{\mathcal L_\alpha}\longleftrightarrow \partial B(0,\lceil\beta a_n\rceil)\right]
\leq r.
\]
As a result, we get
\[
\mathbb P\left[0\stackrel{\mathcal L_\alpha}\longleftrightarrow \partial B(0,\mathcal B_n)\right]
\leq \mathbb P\left[0\stackrel{\mathcal L_\alpha}\longleftrightarrow \partial B(0,\mathcal B_{n-1})\right]\cdot r \leq r^n,
\]
which is precisely \eqref{eq:one arm to mathcalBn}.

\bigskip

Next we prove that there exists $C = C(d)<\infty$ such that for any $\delta\in(0,d-2)$ and $n\geq 1$, 
\begin{equation}\label{eq:moments of Bn}
\mathbb E[\mathcal B_n^\delta]\leq \beta^{\delta n}\cdot C^n\cdot \left(1 + \frac{\alpha}{d-2-\delta}\right)^n .
\end{equation}
Define $\mathcal{G}_{k}=\mathcal{F}_{\mathcal{B}_k}$. Since $\frac{\mathcal{B}_{k+1}}{\mathcal{B}_{k}}$ is $\mathcal{G}_k$-measurable, 
\begin{align*}
 \mathbb{E}\left[\left(\frac{\mathcal{B}_{k+2}}{\mathcal{B}_{k+1}}\right)^{\delta}\Big|\mathcal{G}_k\right]\leq
& (\beta+1)^{\delta}\cdot \mathbb{E}\left[\left(\frac{\mathcal{A}_{k+2}}{\mathcal{B}_{k+1}}\right)^{\delta}\Big|\mathcal{G}_k\right]\\
=&(\beta+1)^{\delta}\cdot \sum\limits_{b_k,b_{k+1}}1_{\{\mathcal{B}_k=b_k,\mathcal{B}_{k+1}=b_{k+1}\}}\mathbb{E}^{B(0,b_k)^c}\left[\left(\frac{\mathcal{A}(b_{k+1})}{b_{k+1}}\right)^{\delta}\right]
\end{align*}
where $\mathbb{P}^{B(0,b_k)^c}$ is the law of the loops avoiding $B(0,b_k)$ and
$$\mathcal{A}(b_{k+1})=\inf\{m: B(0,m-1)\text{ contains all the loops which intersect }B(0,b_{k+1})\}.$$
Since $$\mathbb{E}^{B(0,b_k)^c}\left[\left(\frac{\mathcal{A}(b_{k+1})}{b_{k+1}}\right)^{\delta}\right]\leq \mathbb{E}\left[\left(\frac{\mathcal{A}(b_{k+1})}{b_{k+1}}\right)^{\delta}\right],$$
\begin{multline*}
\mathbb{E}\left[\left(\frac{\mathcal{B}_{k+2}}{\mathcal{B}_{k+1}}\right)^{\delta}\Big|\mathcal{G}_k\right]
\leq (\beta+1)^{\delta}\cdot \sum\limits_{b_k,b_{k+1}}1_{\{\mathcal{B}_k=b_k,\mathcal{B}_{k+1}=b_{k+1}\}}\mathbb{E}\left[\left(\frac{\mathcal{A}(b_{k+1})}{b_{k+1}}\right)^{\delta}\right]\\
\leq (\beta+1)^{\delta}\cdot \sum\limits_{b_k,b_{k+1}}1_{\{\mathcal{B}_k=b_k,\mathcal{B}_{k+1}=b_{k+1}\}}
\left(2^{\delta}+\int\limits_2^{\infty}\,\mathrm{d}\lambda \cdot\delta {\lambda}^{\delta-1}\mathbb{P}[\mathcal{A}(b_{k+1})\geq \lambda b_{k+1}]\right).
\end{multline*}
By Lemma~\ref{lem:one loop connection on Zd}, there exists $C = C(d)<\infty$ such that for all $\lambda>2$, 
\[
\mathbb{P}[\mathcal{A}(b_{k+1})\geq \lambda b_{k+1}]=
\mathbb{P}[\exists \ell\in\mathcal{L}_{\alpha}: B(0,b_{k+1})\stackrel{\ell}\longleftrightarrow\partial B(0,\lceil \lambda b_{k+1}\rceil)]\\
\leq \alpha\cdot C\cdot {\lambda}^{2-d}.
\]
Therefore,
\begin{align*}
\mathbb{E}\left[\left(\frac{\mathcal{B}_{k+2}}{\mathcal{B}_{k+1}}\right)^{\delta}\Big|\mathcal{G}_k\right]\leq 
& (\beta+1)^{\delta}\left(2^{\delta}+\alpha\cdot C \cdot \delta\cdot \int\limits_{2}^{\infty}\mathrm{d}\lambda\cdot  \lambda^{-(d-1-\delta)}\right)\\
=& (\beta +1)^{\delta}\left(2^{\delta}+\frac{\alpha\cdot C\cdot \delta \cdot 2^{2+\delta-d}}{d-2-\delta}\right),
\end{align*}
and \eqref{eq:moments of Bn} follows. 

\bigskip

We can now complete the proof of Theorem~\ref{thm: polynomial decay}. From \eqref{eq:one arm to mathcalBn} and \eqref{eq:moments of Bn}, for $N\geq 1$ and $\epsilon\in(0,d-2)$, 
\begin{align*}
\mathbb{P}[0\overset{\mathcal{L}_{\alpha}}{\longleftrightarrow} \partial B(0,N)]\leq &
\mathbb{P}[0\overset{\mathcal{L}_{\alpha}}{\longleftrightarrow} \partial B(0,\mathcal{B}_{n})]+\mathbb{P}[\mathcal{B}_{n}>N]\notag\\
\leq & r^n+\frac{\mathbb{E}[\mathcal{B}_n^{d-2 - \epsilon}]}{N^{d-2 - \epsilon}}\notag\\
\leq & r^n+\beta^{(d-2) n}\cdot C^n\cdot \left(1 + \frac{\alpha}{\epsilon}\right)^n\cdot N^{2-d + \epsilon}.
\end{align*}
Choosing 
\[
n = \left\lfloor \frac{\epsilon \cdot \log N}{(d-2)\log\beta + \log C + \log (1 + \frac{\alpha}{\epsilon})}\right\rfloor,
\]
we get
\[
\mathbb{P}[0\overset{\mathcal{L}_{\alpha}}{\longleftrightarrow} \partial B(0,N)]\leq
\exp\left\{\log r\cdot \left\lfloor \frac{\epsilon \cdot \log N}{(d-2)\log\beta + \log C + \log (1 + \frac{\alpha}{\epsilon})}\right\rfloor\right\} + N^{2-d + 2\epsilon}.
\]
Note that for any $\alpha<\thran$, there exists $\beta>1$ such that $r<1$. 
Thus, there exist $C = C(d,\alpha)<\infty$ and $c = c(d,\alpha)\in(0,d-2)$ such that 
\[
\mathbb{P}[0\overset{\mathcal{L}_{\alpha}}{\longleftrightarrow} \partial B(0,N)]\leq C\cdot N^{-c}.
\]
Moreover, by Lemma~\ref{lem:uniform upper bound}, there exists $C' = C'(d,\beta)<\infty$ such that for all $\alpha$, $r\leq C'\cdot \alpha$. 
Therefore, one can choose $c(d,\alpha)$ above so that $\lim_{\alpha\to 0}(d-2-c)\cdot \log\frac{1}{\alpha}<\infty$. 
\qed

\section{Loop percolation in dimension \texorpdfstring{$d\geq 5$}{d≥5}}\label{sec: dimension larger than 5}
In this section we prove Theorems~\ref{thm:exponent of one-arm connectivity for d larger than 5} and \ref{thm: asymptotic of critical value in high dimensions}. 
The proof of Theorem~\ref{thm:exponent of one-arm connectivity for d larger than 5} is split into $5$ parts given in $5$ different subsections of this section, 
see, respectively, Propositions \ref{prop:posivity of alpha *}, \ref{prop:exponent for one-arm connectivity for d larger than 5}, 
\ref{prop:two point connectivity}, \ref{prop:exponent of size of cluster for d larger than 5}, and \ref{prop: alpha * is smaller than alpha **}. 
Theorem~\ref{thm: asymptotic of critical value in high dimensions} is restated and proved as Proposition~\ref{prop: asymptotic of critical value in high dimensions} 
in the last subsection. 

\subsection{\texorpdfstring{$\alpha_{\#}>0$}{α\#>0} if and only if \texorpdfstring{$d\geq 5$}{d≥5}}
We prove here that the expected size of $\mathcal C_\alpha(0)$ is finite for small enough $\alpha$ only if $d\geq 5$.
The size of $\mathcal C_\alpha(0)$ is stochastically dominated by the total progeny of 
a Galton-Watson process with offspring distribution defined by the size of 
\[
\mathcal C_\alpha(0,1) \stackrel{\mathrm{def}}= \{x\in\mathbb Z^d\setminus\{0\}~:~\exists\ell\in\mathcal L_\alpha\text{ such that }0\stackrel{\ell}\longleftrightarrow x\}.
\]
Thus, if for some $\alpha>0$, $\mathbb E[\#\mathcal C_\alpha(0,1)]<1$, then the Galton-Watson process is sub-critical, and the expected progeny is finite. 
Existence of such $\alpha$ will follow from the dominated convergence, as soon as we show that for some $\alpha>0$, $\mathbb E[\#\mathcal C_\alpha(0,1)]<\infty$. 
\begin{prop}\label{prop:posivity of alpha *}
$\thrcz>0$ if and only if $d\geq 5$.
\end{prop}
\begin{proof}
Let $d\geq 3$ and $\alpha>0$. We compute
\begin{multline*}
\mathbb{E}[\#\mathcal{C}_{\alpha}(0,1)]=\sum\limits_{x\in\mathbb{Z}^d,x\neq 0}\mathbb{P}[x\in\mathcal{C}_{\alpha}(0,1)]\\
=\sum\limits_{x\in\mathbb{Z}^d,x\neq 0}1-e^{-\alpha\mu(\ell~:~0\stackrel{\ell}\longleftrightarrow x)}
=\sum\limits_{x\in\mathbb{Z}^d,x\neq 0}1-\left(1-\left(\frac{G(0,x)}{G(0,0)}\right)^2\right)^{\alpha},
\end{multline*}
which is finite if and only if $d\geq 5$ by Lemma~\ref{lem:Green function}. 
In particular, if $d= 3$ or $4$, then $\mathbb E[\#\mathcal C_\alpha(0)] \geq \mathbb E[\#\mathcal C_{\alpha}(0,1)] = \infty$ for all $\alpha>0$. 
On the other hand, for $d\geq 5$, by the dominated convergence, there exists $\alpha>0$ such that $\mathbb E[\#\mathcal C_{\alpha}(0,1)]<1$. 
For such $\alpha$, $\#\mathcal C_\alpha(0)$ is dominated by the total progeny of a subcritical Galton-Watson process with offspring distribution 
$\mathbb P[\#\mathcal C_\alpha(0,1)\in\cdot]$. Thus, $\mathbb E[\#\mathcal C_\alpha(0)]<\infty$. 
\end{proof}
\begin{rem}
Domination of the cluster size by the total progeny of a Galton-Watson process is 
used rather often in studies of sub-critical percolation models. 
In the context of loop ensembles, it was used in \cite[Section~2]{Lemaire} 
to study the distribution of connected components of loops on the complete graph. 
\end{rem}

\subsection{One-arm connectivity}
In this section, we prove the second statement in Theorem~\ref{thm:exponent of one-arm connectivity for d larger than 5}, 
which we restate in the following proposition. 
\begin{prop}\label{prop:exponent for one-arm connectivity for d larger than 5}
 For $d\geq 5$ and $\alpha<\thrcz$, there exist constants $0<c(d,\alpha)<C(d,\alpha)<\infty$ such that for all $n$,
$$c(d,\alpha)n^{2-d}\leq\mathbb{P}[0\overset{\mathcal{L}_{\alpha}}{\longleftrightarrow}\partial B(0,n)]\leq C(d,\alpha)n^{2-d}.$$
\end{prop}
We need to introduce the notion of loop distance and decompose the cluster at $0$ according to the loop distance from $0$:
\begin{defn}\label{defn: loop distance and decompostion according to it}
Define a random loop distance $\mathfrak{d}$ on $\mathbb{Z}^d$:
$$\mathfrak{d}(x,y)=\left\{\begin{array}{ll}
                            \inf\left\{k\geq 1: \begin{array}{l}\exists \ell_1,\ldots,\ell_k\in\mathcal{L}_{\alpha}\text{ such that }\\x\in \ell_1,\ell_1\cap \ell_2\neq\phi,\ldots,\ell_{k-1}\cap \ell_k\neq\phi,y\in \ell_k\end{array}\right\} & \text{ for }x\neq y,\\
                            0 & \text{ for }x=y.
                           \end{array}
\right.$$
Then, we decompose $\mathcal{C}_{\alpha}(0)$ into a countable disjoint union:
\begin{equation}\label{def:Calpha0i}
\mathcal{C}_{\alpha}(0)=\bigcup\limits_{i=0}^{\infty}\mathcal{C}_{\alpha}(0,i),\text{ where }\mathcal{C}_{\alpha}(0,i)=\{z\in\mathbb{Z}^d:\mathfrak{d}(0,z)=i\}.
\end{equation}
\end{defn}
\begin{proof}[Proof of Proposition \ref{prop:exponent for one-arm connectivity for d larger than 5}]
The lower bound follows from Theorem~\ref{thm:at most polynomial decay of the connectivity}. 

For $k\geq 0$, set $\mathcal C_k=\bigcup\limits_{i=0}^{k}\mathcal{C}_{\alpha}(0,i)$. Then,
\begin{align*}
\mathbb{P}[0\overset{\mathcal{L}_{\alpha}}{\longleftrightarrow}\partial B(0,2n)]\leq&
\mathbb{P}\left[\mathcal C_k\cap \partial B(0,n)\neq\phi\right]\\
 &+\mathbb{P}\left[\mathcal C_k\cap \partial B(0,n)=\phi,0\overset{\mathcal{L}_{\alpha}}{\longleftrightarrow}\partial B(0,2n)\right].
 \end{align*}
If $\mathcal C_k\cap\partial B(0,n)\neq \phi$, then $\mathcal C_\alpha(0)$ contains a loop from $\mathcal L_\alpha$ with diameter $\geq \frac nk$.
By considering the loop distance $\mathfrak{d}(0,\ell)$ and using the first moment method, we get
\begin{multline*}
\mathbb{P}\left[\mathcal C_k\cap \partial B(0,n)\neq\phi\right]\leq 
\sum\limits_{i=0}^{\infty}\mathbb{E}[\#\mathcal{C}_{\alpha}(0,i)]\cdot \mathbb{P}\left[\exists \ell\in \mathcal{L}_{\alpha}:0\in \ell,Diam(\ell)\geq \frac nk\right]\\
\leq \mathbb{E}[\#\mathcal{C}_{\alpha}(0)]\cdot \mathbb{P}\left[\exists \ell\in\mathcal{L}_{\alpha}:0\stackrel{\ell}\longleftrightarrow\partial B\left(0,\left\lfloor \frac{n}{2k}\right\rfloor\right)\right]
\stackrel{\mathrm{Lemma}~\ref{lem:one loop connection on Zd}}\leq C(d,\alpha)\cdot k^{d-2}\cdot n^{2-d}. 
 \end{multline*}
On the other hand,
\[
\mathbb{P}\left[\mathcal C_k\cap \partial B(0,n)=\phi,0\overset{\mathcal{L}_{\alpha}}{\longleftrightarrow}\partial B(0,2n)\right]
\leq\sum\limits_{x\in B(0,n)}\mathbb{P}[x\in\mathcal{C}_{\alpha}(0,k),x\overset{(\mathcal{L}_{\alpha})^{(\mathcal C_{k-1})^c}}{\longleftrightarrow}\partial B(0,2n)].
\]
Since $\{x\in\mathcal{C}_{\alpha}(0,k)\}$ is $\mathcal{F}_{\mathcal C_{k-1}}$ measurable and 
\[
\mathbb{P}[x\overset{(\mathcal{L}_{\alpha})^{(\mathcal C_{k-1})^c}}{\longleftrightarrow}\partial B(0,2n)|\mathcal{F}_{\mathcal C_{k-1}}]
\leq \mathbb{P}[x\overset{\mathcal{L}_{\alpha}}{\longleftrightarrow}\partial B(0,2n)]
\leq\mathbb{P}[0\overset{\mathcal{L}_{\alpha}}{\longleftrightarrow}\partial B(0,n)],
\]
we get
\[
\mathbb{P}\left[\mathcal C_k\cap \partial B(0,n)=\phi,0\overset{\mathcal{L}_{\alpha}}{\longleftrightarrow}\partial B(0,2n)\right]
\leq \mathbb{E}[\#\mathcal{C}_{\alpha}(0,k)]\cdot \mathbb{P}[0\overset{\mathcal{L}_{\alpha}}{\longleftrightarrow}\partial B(0,n)].
\]
Putting two bounds together, 
\[
\mathbb{P}[0\overset{\mathcal{L}_{\alpha}}{\longleftrightarrow}\partial B(0,2n)]\leq 
C(d,\alpha)\cdot k^{d-2}\cdot n^{2-d}+\mathbb{E}[\#\mathcal{C}_{\alpha}(0,k)]\cdot \mathbb{P}[0\overset{\mathcal{L}_{\alpha}}{\longleftrightarrow}\partial B(0,n)].
\]
We choose $k=k_0$ large enough such that $\mathbb{E}[\#\mathcal{C}_{\alpha}(0,k_0)]\leq 2^{1-d}$ and take
$C'(d,\alpha)=\max(2^{d-1}k_0^{d-2}C(d,\alpha),1)$. Then, by induction on $n$, 
\[
\mathbb{P}[0\overset{\mathcal{L}_{\alpha}}{\longleftrightarrow}\partial B(0,2^n)]\leq C'(d,\alpha)\cdot (2^n)^{2-d}.
\]
The proof is complete by the monotonicity of $\mathbb{P}[0\overset{\mathcal{L}_{\alpha}}{\longleftrightarrow}\partial B(0,n)]$ in $n$.
\end{proof}

\medskip

\begin{rem}\ \label{rem:remark for d larger than 5}
Recall from Theorem~\ref{thm:at most polynomial decay of the connectivity} and Lemma~\ref{lem:one loop connection on Zd} that 
the probability that a single loop from $\mathcal L_\alpha$ passing through a given vertex $x\in\mathbb Z^d$ has diameter $\geq n$ is of order $n^{2-d}$. 
Thus, for $d\geq 5$ and $\alpha<\thrcz$, 
the probability that $\mathcal C_\alpha(0)$ contains a loop of diameter $\geq n$ is 
of the same order as the probability of one arm to $\partial B(0,n)$. 
This suggests that long connections in $\mathcal L_\alpha$ for $d\geq 5$ and $\alpha<\thrcz$ arise because of a single big loop. 
It is indeed the case, as one can show that the probability that $\mathcal C_\alpha(0)$ contains at least two loops of diameter $\geq m$ 
is $O(m^{6-2d})$, and the probability of having a path from $0$ to $\partial B(0,n)$ only through loops of $\mathcal L_\alpha$ of diameter $\leq m$ is $O(e^{-c(d,\alpha)\cdot \frac nm})$. 
Since we do not use such refined estimates, we omit details of their proofs. 
Curiously, as we will see later, the situation in dimensions $d=3$ and $4$ is rather different, as long connections through small loops are more likely than 
connections through a single big loop. 
\end{rem}
\subsection{Two point connectivity}
In this section we prove the third statement of Theorem~\ref{thm:exponent of one-arm connectivity for d larger than 5} about the bounds on the two point connectivity, 
which we restate in the next proposition. 
As in the case of one arm connectivity, the lower bound here is given by one loop connection. 
While the upper bound is obtained by one loop connection together with the upper bound for the decay of one arm connectivity.
\begin{prop}\label{prop:two point connectivity}
For $d\geq 5$ and $\alpha<\thrcz$,
\[
0<\inf\limits_{x\in\mathbb{Z}^d}\mathbb{P}[x\in\mathcal{C}_{\alpha}(0)](||x||_{\infty}+1)^{2(d-2)}\leq \sup\limits_{x\in\mathbb{Z}^d}\mathbb{P}[x\in\mathcal{C}_{\alpha}(0)](||x||_{\infty}+1)^{2(d-2)}<\infty.
\]
\end{prop}

\begin{proof}
Since $\mathbb{P}[x\in\mathcal{C}_{\alpha}(0)]\geq\mathbb{P}[\exists \ell\in\mathcal{L}_{\alpha}:0\stackrel{\ell}\longleftrightarrow x]$, 
the lower bound is given by Lemmas \ref{lem:non-trivial loop visiting F} and \ref{lem:Green function}. It remains to show the upper bound. 

Let $n=||x||_{\infty}$. Without loss of generality we may suppose $n\geq 3$. We divide the loops $\mathcal{L}_{\alpha}$ into four independent set of loops as follows:
\begin{itemize}
 \item $\mathcal{L}_{1,1}\overset{\text{def}}{=}\{\ell\in\mathcal{L}_{\alpha}:\ell\text{ intersects }B(0,\lfloor n/3\rfloor)\text{ and }B(x,\lfloor n/3\rfloor)\}$,
 \item $\mathcal{L}_{1,0}\overset{\text{def}}{=}\{\ell\in\mathcal{L}_{\alpha}:\ell\text{ intersects }B(0,\lfloor n/3\rfloor)\text{ but not }B(x,\lfloor n/3\rfloor)\}$,
 \item $\mathcal{L}_{0,1}\overset{\text{def}}{=}\{\ell\in\mathcal{L}_{\alpha}:\ell\text{ intersects }B(x,\lfloor n/3\rfloor)\text{ but not }B(0,\lfloor n/3\rfloor)\}$,
 \item $\mathcal{L}_{0,0}\overset{\text{def}}{=}\{\ell\in\mathcal{L}_{\alpha}:\ell\text{ avoids }B(0,\lfloor n/3\rfloor)\text{ and }B(x,\lfloor n/3\rfloor)\}$.
\end{itemize}
Let $\mathcal C_\alpha^k(z)$ be the cluster of $z$ induced by the loops of $\mathcal L_\alpha$ which are entirely contained in $B(z,k)$. 
The main observation is that when $x\in\mathcal{C}_{\alpha}(0)$, at least one of the four events occurs:
\begin{itemize}
\item 
$E_1\overset{\text{def}}{=}\left\{0\overset{\mathcal{L}_{1,0}}{\longleftrightarrow}\partial B(0,\lfloor n/3\rfloor)\right\}
\cap\left\{x\overset{\mathcal{L}_{0,1}\cup\mathcal{L}_{1,1}}{\longleftrightarrow}\partial B(x,\lfloor n/3\rfloor)\right\}$,
\item 
$E_2\overset{\text{def}}{=}\left\{0\overset{\mathcal{L}_{1,0}\cup\mathcal{L}_{1,1}}{\longleftrightarrow}\partial B(0,\lfloor n/3\rfloor)\right\}
\cap\left\{x\overset{\mathcal{L}_{0,1}}{\longleftrightarrow}\partial B(x,\lfloor n/3\rfloor)\right\}$,
\item 
$E_3\overset{\text{def}}{=}
\left\{\exists a\in \mathcal C_\alpha^{\lfloor n/3\rfloor}(0),~b\in \mathcal C_\alpha^{\lfloor n/3\rfloor}(x),~\ell\in\mathcal L_{1,1}~:~a\stackrel{\ell}\longleftrightarrow b\right\}$,

\item
$E_4\overset{\text{def}}{=}\left\{\exists a\in \mathcal C_\alpha^{\lfloor n/3\rfloor}(0),~b\in \mathcal C_\alpha^{\lfloor n/3\rfloor}(x),~\ell_1,\ell_2\in\mathcal L_{1,1}~:~
\substack{a\in\ell_1,~\ell_1\cap \mathcal C_\alpha^{\lfloor n/3\rfloor}(x) = \phi,\\
b\in\ell_2,~\ell_2\cap \mathcal C_\alpha^{\lfloor n/3\rfloor}(0) = \phi}\right\}$.

\end{itemize}
Thus, $\mathbb{P}[x\in\mathcal{C}_{\alpha}(0)]\leq\mathbb{P}[E_1]+\mathbb{P}[E_2]+\mathbb{P}[E_3]+\mathbb P[E_4]$. 
By independence of $\mathcal{L}_{0,1},\mathcal{L}_{1,0}$, and $\mathcal{L}_{1,1}$, translation invariance of $\mathcal L_\alpha$, and Proposition~\ref{prop:exponent for one-arm connectivity for d larger than 5},
\[
\mathbb{P}[E_1]=\mathbb{P}[E_2]\leq(\mathbb{P}[0\overset{\mathcal{L}_{\alpha}}{\longleftrightarrow}\partial B(0,\lfloor n/3\rfloor)])^2\leq C_1(d,\alpha)\cdot n^{2(2-d)}. 
\]
Next we estimate $\mathbb P[E_3]$:
\begin{align*}
\mathbb{P}[E_3]\leq &\sum\limits_{\substack{a\in B(0,\lfloor n/3\rfloor),\\b\in B(x,\lfloor n/3\rfloor)}}
\mathbb{P}[\exists \ell\in \mathcal{L}_{1,1}:a\stackrel{\ell}\longleftrightarrow b]\cdot 
\mathbb{P}[0\overset{\mathcal{L}_{1,0}}{\longleftrightarrow}a]\cdot \mathbb{P}[x\overset{\mathcal{L}_{0,1}}{\longleftrightarrow}b]\\
\leq &(\mathbb{E}[\#\mathcal{C}_{\alpha}(0)])^2\cdot \max\limits_{\substack{a\in B(0,\lfloor n/3\rfloor),\\b\in B(x,\lfloor n/3\rfloor)}}
\mathbb{P}[\exists \ell\in\mathcal{L}_{\alpha}:a\stackrel{\ell}\longleftrightarrow b]\\
\leq &\alpha\cdot (\mathbb{E}[\#\mathcal{C}_{\alpha}(0)])^2\cdot \max\limits_{\substack{a\in B(0,\lfloor n/3\rfloor),\\b\in B(x,\lfloor n/3\rfloor)}}\mu(\ell~:~a\stackrel{\ell}\longleftrightarrow b).
\end{align*}
Since $||a-b||_{\infty}\geq n/3$ for all $a\in B(0,\lfloor n/3\rfloor)$ and $b\in B(x,\lfloor n/3\rfloor)$, 
by Lemmas \ref{lem:non-trivial loop visiting F} and \ref{lem:Green function}, $\mathbb{P}[E_3]\leq C_2(d,\alpha)\cdot n^{2(2-d)}$.

Finally, we estimate $\mathbb P[E_4]$ by first conditioning on $\mathcal C_\alpha^{\lfloor n/3\rfloor}(0)$ and $\mathcal C_\alpha^{\lfloor n/3\rfloor}(x)$, 
and then using independence between loops from $\mathcal L_{1,1}$ that intersect $\mathcal C_\alpha^{\lfloor n/3\rfloor}(0)$ and do not intersect it:
\begin{multline*}
\mathbb P[E_4]\leq \mathbb E[\#\mathcal C_\alpha^{\lfloor n/3\rfloor}(0)]\cdot 
\mathbb E[\#\mathcal C_\alpha^{\lfloor n/3\rfloor}(x)] \cdot
\mathbb P[\exists \ell\in\mathcal L_\alpha~:~0\stackrel{\ell}\longleftrightarrow \partial B(0,\lfloor n/3\rfloor)]^2\\
\stackrel{\mathrm{Lemmas}~\ref{lem:non-trivial loop visiting F},~\ref{lem:Green function}}\leq C_3(d,\alpha)\cdot n^{2(2-d)}.
\end{multline*}

Thus, for $||x||_{\infty}\geq 3$, $\mathbb{P}[x\in\mathcal C_\alpha(0)]\leq (2C_1+C_2+C_3)\cdot ||x||^{2(2-d)}$.
\end{proof}

\medskip
\begin{rem}
As in the case of one arm connectivity, see Remark~\ref{rem:remark for d larger than 5}, 
one can show that for $d\geq 5$ and $\alpha<\thrcz$, the most likely situation for $0\stackrel{\mathcal L_\alpha}\longleftrightarrow x$ is to have a large loop which passes near $0$ and near $x$, 
i.e., the existence of connections between $0$ and $x$ with two large loops or with only small loops are both of probability $o(||x||^{2(2-d)})$. 
\end{rem}

\subsection{Tail of the cluster size}
In this section we prove the fourth statement of Theorem~\ref{thm:exponent of one-arm connectivity for d larger than 5}, 
showing that the tail of the distribution of $\#\mathcal C_\alpha(0)$ is of order $n^{1-d/2}$, see Proposition \ref{prop:exponent of size of cluster for d larger than 5}. 
The lower bound is given by the loops passing through $0$. 
Roughly speaking, the upper bound is given by the total progeny of a sub-critical Galton-Watson process which dominates the cluster size. 
The existence of such sub-critical Galton-Watson process is guaranteed by assumption $\alpha<\thrcz$. 
An upper bound for the sub-critical Galton-Watson process is given in Lemma \ref{lem:upper bound of the tail of total progeny of our Galton-Watson tree}. 
Later we will take the offspring distribution to be the distribution of $\#\mathcal{U}_{\alpha}(0,K)$ where for $K\geq 1$,
\begin{equation}\label{def:U0K}
 \#\mathcal{U}_{\alpha}(0,K)\overset{\text{def}}=\left\{x\in\mathbb{Z}^d:
 \begin{array}{l}
  \exists \ell_1,\ldots,\ell_K\in\mathcal{L}_{\alpha}\text{ such that }0\in \ell_1,x\in \ell_K\\
  \text{and }\ell_i\cap \ell_{j}\neq\phi\text{ iff }|i-j|\leq 1
 \end{array}
 \right\}.
\end{equation}
The crucial point is that for $x,y\in\mathbb{Z}^d$, the following is an increasing event:
\[
\left\{\exists \ell_1,\ldots,\ell_K\in\mathcal{L}_{\alpha}:x\in \ell_1,y\in \ell_K\text{ and that }\ell_i\cap \ell_{j}\neq\phi\text{ iff }|i-j|\leq 1\right\}.
\]
This enables us to dominate $(\#\mathcal{C}_{\alpha}(0,Ki))_{i\geq 0}$ by a sub-critical Galton-Watson process with offspring $\#\mathcal{U}_{\alpha}(0,K)$. In order to apply Lemma \ref{lem:upper bound of the tail of total progeny of our Galton-Watson tree} to dominate the total progeny, we need an upper bound estimate for the tail of $\mathcal{U}_{\alpha}(0,K)$. This is given in Lemma \ref{lem:bounds for C(d,K)} .

\begin{prop}\label{prop:exponent of size of cluster for d larger than 5}
For $d\geq 5$ and $\alpha<\thrcz$, there exist $0<c(d,\alpha)<C(d,\alpha)<\infty$ such that for all $n$,
\begin{equation}\label{eq:bound for the size of cluster}
 c(d,\alpha)n^{1-d/2}<\mathbb{P}[\#\mathcal{C}_{\alpha}(0)>n]<C(d,\alpha)n^{1-d/2}.
\end{equation}
\end{prop}
The proof of the proposition is based on the following lemmas. 
\begin{lem}\label{lem: tail of one generation GW}
Suppose $\xi$ and $\eta$ are $\mathbb{N}$ valued variables with finite means. 
Denote by $\bar{F}$ the tail of the distribution function of $\xi$ and by $\bar{G}$ that of $\eta$. 
Suppose $\bar{F}(x),\bar{G}(x)\leq x^{-a}h(x)$ where $a>1$ and $h$ slowly varies as $x\rightarrow\infty$. 
Take a sequence $(\eta_i)_{i\geq 0}$ of independent copies of $\eta$ which is also independent of $\xi$. 
Then there exists $C<\infty$ such that for $n\geq 1$,
\[
\mathbb{P}\left[\sum\limits_{i=1}^{\xi}\eta_i> n\right]\leq Cn^{-a}h(n).
\]
\end{lem}
\begin{lem}\label{lem:upper bound of the tail of total progeny of our Galton-Watson tree}
Let $\bar F(x) = 1 - F(x)$ be the tail of a distribution function $F$, and suppose that 
$\bar{F}(x)\leq x^{-a}h(x)$ where $a>1$ and $h$ is slowly varying when $x\rightarrow\infty$.
For a sub-critical Galton-Watson process $(Z_n)_{n\geq 0}$ with offspring distribution $F$, let $S_n=\sum\limits_{i=0}^{n}Z_i$. 
Then there exists a constant $C<\infty$ such that
\[
\mathbb{P}[S_{\infty}>n]<C n^{-a}h(n).
\]
\end{lem}
Recall the definition of the partition $(\mathcal C_\alpha(0,i),i\geq 0)$ of $\mathcal C_\alpha(0)$ from \eqref{def:Calpha0i}. 
\begin{lem}\label{lem:bounds for C(d,K)}
 For $d\geq 5$ and $K\geq 1$, there exist $0<c(d,\alpha)\leq C(d,\alpha,K)<\infty$ such that
\begin{equation}\label{eq:lbfC1}
c(d,\alpha)\cdot n^{1-d/2}\leq \mathbb P[\#\mathcal C_\alpha(0,1)>n]\leq \mathbb{P}\left[\bigcup\limits_{i=1}^{K}\#\mathcal{C}_{\alpha}(0,K)> n\right]\leq C(d,\alpha,K)\cdot n^{1-d/2}.
\end{equation}
 Since $\mathcal{U}_{\alpha}(0,K)\subset\bigcup\limits_{i=1}^{K}\#\mathcal{C}_{\alpha}(0,K)$, the same upper bound holds for the tail distribution of $\#\mathcal{U}_{\alpha}(0,K)$:
 \begin{equation}\label{eq:lbfC2}
  \mathbb{P}[\#\mathcal{U}_{\alpha}(0,K)> n]\leq C(d,\alpha,K)\cdot n^{1-d/2}.
 \end{equation}
\end{lem}

We postpone the proof of the lemmas until the end of this section.

\begin{proof}[Proof of Proposition \ref{prop:exponent of size of cluster for d larger than 5}]
The lower bound follows from Lemma~\ref{lem:bounds for C(d,K)}. 

For the upper bound, let $d\geq 5$ and $\alpha<\thrcz$. 
Since $\#\mathcal{C}_{\alpha}(0)$ is finite almost surely, $\lim\limits_{K\rightarrow\infty}\#\mathcal{U}_{\alpha}(0,K)=0$. 
By the dominated convergence, we can choose $K$ large enough such that $\mathbb{E}[\#\mathcal{U}_{\alpha}(0,K)]<1$. 
Then we define a sub-critical Galton-Watson process $(Z_i)_{i\geq 0}$ with offspring distribution $\mathbb{P}[\#\mathcal{U}_{\alpha}(0,K)\in\cdot]$, 
so that $\sum\limits_{i\geq 0}\#\mathcal{C}_{\alpha}(0,Ki)$ is stochastically dominated by $S_{\infty}\overset{\text{def}}{=}\sum\limits_{i=0}^{\infty}Z_i$.
By Lemma \ref{lem:bounds for C(d,K)} and Lemma \ref{lem:upper bound of the tail of total progeny of our Galton-Watson tree}, there exists $C=C(d,\alpha)<\infty$ such that
\[
\mathbb{P}\left[\sum\limits_{i\geq 0}\#\mathcal{C}_{\alpha}(0,Ki)>n\right]\leq \mathbb{P}[S_{\infty}>n]<C \cdot n^{1-d/2}.
\]
Let $(\eta_i)_i$ be a sequence of independent copies of $S_{\infty}$. 
We further suppose that they are independent of $\mathcal{L}_{\alpha}$. 
Then for $j=1,\ldots,K-1$, $\sum\limits_{i\geq 0}\#\mathcal{C}_{\alpha}(0,Ki+j)$ is stochastically dominated by $\sum\limits_{i=1}^{\#\mathcal{C}_{\alpha}(0,j)}\eta_i$. 
By applying Lemma \ref{lem: tail of one generation GW} for $\sum\limits_{i=1}^{\#\mathcal{C}_{\alpha}(0,j)}\eta_i$, there exists $C'=C'(d,\alpha)<\infty$ such that for $j=1,\ldots,K-1$,
\[
\mathbb{P}\left[\sum\limits_{i\geq 0}\#\mathcal{C}_{\alpha}(0,Ki+j)>n\right]\leq \mathbb{P}\left[\sum\limits_{i=1}^{\#\mathcal{C}_{\alpha}(0,j)}\eta_i>n\right]<C' \cdot n^{1-d/2}.
\]
Similar upper bound with a bigger constant holds for $\#\mathcal{C}_{\alpha}(0)$, since
\[
\{\#\mathcal{C}_{\alpha}(0)>n\}\subset\bigcup\limits_{j=0}^{K-1}\left\{\sum\limits_{i=0}^{\infty}\#\mathcal{C}_{\alpha}(0,Ki+j)>\frac{n}{K}\right\}.
\]
The proof is complete.
\end{proof}

It remains to prove the lemmas.

\subsubsection{Proof of Lemma~\ref{lem: tail of one generation GW}}

By \cite[Theorem~2]{NagaevMR616627}, for $\gamma>0$,
\[
\sup\limits_{\{k,x:x\geq \gamma k\}}\frac{\mathbb{P}\left[\sum\limits_{i=1}^{k}\eta_i-k\mathbb{E}[\eta]>x\right]}{k\mathbb{P}[\eta>x]}<\infty.
\]
By taking $\gamma=\mathbb{E}[\eta]$, $k=\xi$ and $x=n/2$, there exists $C<\infty$ such that
\[
\mathbb{P}\left[\sum\limits_{i=1}^{\xi}\eta_i>n,\xi\leq\frac{n}{2\mathbb{E}[\eta]}\right]\leq C\cdot \mathbb{E}[\xi]\cdot \mathbb{P}[\eta>n/2]
\leq C\cdot 2^a\cdot \mathbb{E}[\xi]\cdot n^{-a}h(n/2).
\]
Since $\mathbb{P}[\xi>\frac{n}{2\mathbb{E}[\eta]}]\leq (2\mathbb{E}[\eta])^an^{-a}h\left(\frac{n}{2\mathbb{E}[\eta]}\right)$ by assumption,
\[
\mathbb{P}\left[\sum\limits_{i=1}^{\xi}\eta_i>n\right]\leq 2^a\cdot C\cdot \mathbb{E}[\xi]\cdot n^{-a}h(n/2)+(2\mathbb{E}[\eta])^a\cdot n^{-a}h\left(\frac{n}{2\mathbb{E}[\eta]}\right).
\]
By the definition of slowly varying function, there exists $C'<\infty$ such that
\[
\mathbb{P}\left[\sum\limits_{i=1}^{\xi}\eta_i>n\right]\leq C'\cdot n^{-a}h(n).
\]
\qed

\subsubsection{Proof of Lemma~\ref{lem:upper bound of the tail of total progeny of our Galton-Watson tree}}

Set $H(x)=x^{-a}h(x)$. We may assume that $H(x)\leq 1$ for $0<x \leq 1$.
Denote by $m\overset{\text{def}}{=}\mathbb{E}[Z_1]$. 
By assumption, $m<1$, hence $\mathbb{E}[Z_{k-1}]=m^{k-1}$ and $\mathbb{P}[Z_{k-1}>0]\leq \mathbb{E}[Z_{k-1}]=m^{k-1}$. 

Fix $\delta>0$ and take $\rho\in]m,1[$ close enough to $1$ so that $m\rho^{-a-\delta}<1$. (This particular choice of $\rho$ will be clear during the proof.) 
Let $\gamma=\rho-m$.  
Then, 
\begin{equation}\label{eq:lubottotpooGt1}
\mathbb{P}\left[Z_k\geq \rho^k n,Z_{k-1}<\rho^{k-1} n\right]
\leq\mathbb{P}\left[Z_k-mZ_{k-1}\geq \gamma\rho^{k-1}n>\gamma Z_{k-1},Z_{k-1}>0\right].
\end{equation}
 By \cite[Theorem 2]{NagaevMR616627}, for $\gamma>0$,
 \begin{equation}\label{eq:lubottotpooGt2}
  C(\gamma)\overset{\text{def}}{=}\sup\limits_{p\geq 1}\sup\limits_{x\geq \gamma p}\frac{\mathbb{P}\left[\sum\limits_{i=1}^{p}\eta_i-pm>x\right]}{pH(x)}<\infty
 \end{equation}
 where $(\eta_i)_i$ are i.i.d. variables with the distribution $F$. By conditioning on $Z_{k-1}$ and then applying \eqref{eq:lubottotpooGt2} with $p=Z_{k-1}$ and $x=\gamma\rho^{k-1}n$, 
 \begin{equation}\label{eq:lubottotpooGt3}
 \eqref{eq:lubottotpooGt1}\leq  C(\gamma)\cdot \mathbb{E}[Z_{k-1}]\cdot H(\gamma\rho^{k-1}n)
 =C(\gamma)\cdot m^{k-1}\cdot H(\gamma\rho^{k-1}n).
 \end{equation}
By Protter's Theorem, see \cite[Theorem 1.5.6]{BinghamMR1015093}, 
there exists $C'= C'(\delta)$ such that for all $n\geq 1$ and $c\in]0,1[$, $H(cn)\leq C'\cdot c^{-a-\delta}\cdot H(n)$. 
Therefore, 
\[
\eqref{eq:lubottotpooGt3}\leq C'(\delta)C(\gamma)\gamma^{-a-\delta}(m\rho^{-a-\delta})^{k-1}H(n).
\]
For $n>1/\rho$, $Z_0=1<\rho n$. Thus,
\[
\left\{S_{\infty}>\frac{1}{1-\rho}n\right\}\subset\{\exists k\geq 1:Z_k\geq \rho^kn,Z_{k-1}<\rho^{k-1}n\},
\]
and for $n>1/\rho$,
\begin{align*}
\mathbb{P}\left[S_{\infty}>\frac{\rho}{1-\rho}n\right]\leq & \mathbb{P}[\exists k\geq 1:Z_k\geq \rho^kn,Z_{k-1}<\rho^{k-1}n]\\
\leq &\sum\limits_{k\geq 1}\mathbb{P}[Z_k\geq \rho^kn,Z_{k-1}<\rho^{k-1}n]\\
\leq &\frac{C'(\delta)C(\rho-m)}{(\rho-m)^{a+\delta}(1-m\rho^{-a-\delta})}H(n).
\end{align*}
Finally, there exists $C''<\infty$ such that for $n\geq 1$, $\mathbb{P}[S_{\infty}>n]<C''\cdot H(n)$.
\qed

\subsubsection{Proof of Lemma \ref{lem:bounds for C(d,K)}}

Lemma~\ref{lem:bounds for C(d,K)} follows from the lemma:
\begin{lem}\label{lem:tail of the range of the loop intersecting 0}
For $d\geq 3$ and $\alpha>0$, as $x\rightarrow\infty$,
$$\mathbb{P}[\#\mathcal{C}_{\alpha}(0,1)>x]\sim \frac{\alpha d^{d/2}}{(d/2-1)(2\pi G(0,0))^{d/2}}x^{1-\frac{d}{2}}.$$
In particular, for $p\geq 0$,
$\mathbb{E}[(\#\mathcal{C}_{\alpha}(0,1))^{p}]<\infty$ iff $p<\frac{d}{2}-1$.
\end{lem}
Indeed, the case $K=1$ in Lemma~\ref{lem:bounds for C(d,K)} follows from Lemma~\ref{lem:tail of the range of the loop intersecting 0}. 
We prove the general case by induction on $K$. 
Suppose that for any $K\leq m$ there exists $C(d,\alpha,K)<\infty$ such that 
\begin{equation}\label{eq:lbfC3}
\mathbb{P}\left[\#\mathcal{C}_{\alpha}(0,K)> n\right]\leq C(d,\alpha,K)\cdot n^{1-d/2}.
\end{equation}
Let $(\eta_i)_{i}$ be a sequence of independent variables with distribution $\mathbb{P}[\#\mathcal{C}_{\alpha}(0,1)\in \cdot]$ which is also independent from $\mathcal{L}_{\alpha}$. 
Then, $\#\mathcal{C}_{\alpha}(0,m+1)$ is stochastically dominated by $\sum\limits_{i=1}^{\#\mathcal{C}_{\alpha}(0,m)}\eta_i$. 
The proof is complete by using \eqref{eq:lbfC3} for $K=1$ and $K=m$, and by applying Lemma~\ref{lem: tail of one generation GW} with $\xi=\#\mathcal{C}_{\alpha}(0,m)$.
\qed

\medskip

It remains to prove Lemma~\ref{lem:tail of the range of the loop intersecting 0}. 
We first prove a result on the tail of the range of the first finite excursion of a SRW at $0$.

\begin{lem}\label{lem:tail of the length of first finite excursion}
Let $\mathbb{P}^{0}$ be the law of SRW on $\mathbb{Z}^d$ starting from $0$. 
Let $\mathbb{P}^{ex}$ be the law of the first excursion under $\mathbb{P}^{0}[\cdot|\tau^{+}(0)<\infty]$. 
Let $\mathbb{P}^{ex,2n}$ be the law of the first excursion given that the first excursion has exactly $2n$ jumps. 
Set $F=\mathbb{P}^{0}[\tau^{+}(0)<\infty]=1-\frac{1}{G(0,0)}$. Then,
\begin{itemize}
 \item[a)] for any $\epsilon>0$,
$$\lim\limits_{n\rightarrow\infty}\mathbb{P}^{ex,2n}\left[\left|\frac{1}{2n}\#(\text{Range of excursion})-(1-F)\right|>\epsilon\right]=0;$$
 \item[b)] $\mathbb{P}^{ex}[\#(\text{Range of excursion})>x]\sim \frac{d^{d/2}(1-F)^{d/2+1}}{(d/2-1)(2\pi)^{d/2}F}x^{1-\frac{d}{2}}$.
\end{itemize}

\end{lem}
\begin{proof}\
Firstly, by \cite{GriffinMR1060695}, for $d\geq 3$, as $n\rightarrow\infty$,
$$\mathbb{P}^0[\tau^{+}(0)=2n]\sim (1-F)^2\mathbb{P}^0[X_{2n}=0].$$
Then, by \cite[Theorem 1.2.1]{LawlerMR2985195},
$$\mathbb{P}^0[X_{2n}=0]\sim \frac{2d^{d/2}}{(4\pi n)^{d/2}}\text{ as }n\rightarrow\infty.$$
Thus,
\begin{equation}\label{eq:Griffin}
 \mathbb{P}^{0}[\tau^{+}(0)=2n]\sim (1-F)^2 \frac{2d^{d/2}}{(4\pi n)^{d/2}}.
\end{equation}
\begin{itemize}
 \item[a)] We see that $\mathbb{P}^{0}[\tau^{+}(0)=2n]$ decays polynomially. By \cite[Theorem 1]{HamanaKestenMR1841327},
 $$\mathbb{P}^{0}\left[\frac{1}{2n}\#(\text{Range of }\{X_0,\ldots,X_{2n}\})>1-F+\epsilon\right]\text{ goes to }0\text{ exponentially fast}.$$
 Then,
 \begin{multline*}
  \mathbb{P}^{ex,2n}\left[\frac{1}{2n}\#(\text{Range of excursion})>1-F+\epsilon\right]\\
  =\frac{1}{\mathbb{P}^{0}[\tau^{+}(0)=2n]}\mathbb{P}^{0}\left[\frac{1}{2n}\#(\text{Range of }\{X_0,\ldots,X_{2n}\})>1-F+\epsilon,\tau^{+}(0)=2n\right]\\
  \leq\frac{1}{\mathbb{P}^{0}[\tau^{+}(0)=2n]}\mathbb{P}^{0}\left[\frac{1}{2n}\#(\text{Range of }\{X_0,\ldots,X_{2n}\})>1-F+\epsilon\right]
 \end{multline*}
 which also goes to $0$ exponentially fast.
 On the other hand, for any fixed $\delta>0$, the Radon-Nikodym derivative $\frac{d\mathbb{P}^{ex,2n}}{d\mathbb{P}^{0}}$ with respect to $\sigma(X_0,\ldots,X_{\lfloor 2n-n\delta\rfloor})$ is bounded by a constant $c(\epsilon)$:
 \begin{align*}
  \mathbb{E}^0\left[\frac{d\mathbb{P}^{ex,2n}}{d\mathbb{P}^{0}}\Big|\sigma(X_0,\ldots,X_{\lfloor 2n-n\delta\rfloor})\right]=&1_{\{\tau^{+}(0)>\lfloor 2n-n\delta\rfloor\}}\frac{\mathbb{P}^{X_{\lfloor 2n-n\delta\rfloor}}\left[\tau^{+}(0)=\lceil n\delta\rceil\right]}{\mathbb{P}^0[\tau^{+}(0)=2n]}\\
 \leq & \frac{\mathbb{P}^{X_{\lfloor 2n-n\delta\rfloor}}\left[X_{\lceil n\delta\rceil}=0\right]}{\mathbb{P}^0[\tau^{+}(0)=2n]}\\
 \leq & c(\delta)
 \end{align*}
 where the last inequality comes from local central limit theorem with Equation \eqref{eq:Griffin}.

 For any fixed $\epsilon>0$, choose $\delta$ small enough such that $(1-F-\epsilon)\frac{2}{2-\delta}<1-F-\epsilon/2$. Then, for $n$ large enough,
 \begin{align*}
  \mathbb{P}^{ex,2n} & \left[\frac{1}{2n}\#(\text{Range of excursion})<1-F-\epsilon\right]\\
  & \leq \mathbb{P}^{ex,2n}\left[\frac{1}{2n}\#(\text{Range of }\{X_0,\ldots,X_{\lfloor 2n-n\delta\rfloor}\})<1-F-\epsilon\right]\\
  & \leq \mathbb{P}^{ex,2n}\left[\frac{1}{\lfloor 2n-n\delta\rfloor}\#(\text{Range of }\{X_0,\ldots,X_{\lfloor 2n-n\delta\rfloor}\})<1-F-\epsilon/2\right]\\
  & \leq c(\epsilon)\mathbb{P}^{0}\left[\frac{1}{\lfloor 2n-n\delta\rfloor}\#(\text{Range of }\{X_0,\ldots,X_{\lfloor 2n-n\delta\rfloor}\})<1-F-\epsilon/2\right]
 \end{align*}
 which goes to $0$ as $n$ tends to infinity since under $\mathbb{P}^0$,
 $$\frac{1}{n}\#(\text{Range of }\{X_0,\ldots,X_n\})\overset{n\rightarrow\infty}{\longrightarrow} 1-F\text{ in probability},$$
 see e.g. \cite[T1 in Section 4]{SpitzerMR0388547}.
 \item[b)] By decomposing the event according to the value of $\tau^{+}(0)$, we see that as $x\rightarrow\infty$,
 \begin{align*}
  \mathbb{P}^{ex}[\#(\text{Range of excursion})>x]=& \sum\limits_{n\geq \lfloor x/2\rfloor}\mathbb{P}^{0}[\tau^{+}(0)=2n]/\mathbb{P}^{0}[\tau^{+}(0)<\infty]\\
  & \times\mathbb{P}^{ex,2n}[\#(\text{Range of excursion})>x].
 \end{align*}
 By using the result in the first part and Equation \eqref{eq:Griffin}, as $x\rightarrow\infty$
 \begin{align*}
  \mathbb{P}^{ex}[\#(\text{Range of excursion})>x]\sim& \sum\limits_{n\geq \frac{x}{2(1-F)}}\mathbb{P}^{0}[\tau^{+}(0)=2n]/\mathbb{P}^{0}[\tau^{+}(0)<\infty]\\
  \sim&\sum\limits_{n\geq \frac{x}{2(1-F)}}\frac{(1-F)^2}{F} \frac{2d^{d/2}}{(4\pi n)^{d/2}}\\
  \sim&\frac{d^{d/2}(1-F)^{d/2+1}}{(d/2-1)(2\pi)^{d/2}F}x^{1-\frac{d}{2}}.\qedhere
 \end{align*}
\end{itemize}
\end{proof}

\medskip

\begin{proof}[Proof of Lemma \ref{lem:tail of the range of the loop intersecting 0}]
By \cite[Theorem 3.29]{StanMR3097424}, the result of Lemma \ref{lem:tail of the length of first finite excursion} implies that 
the distribution of $\#(\text{Range of excursion})$ under $\mathbb{P}^{ex}$ is sub-exponential for $d\geq 3$. 
By \cite[Theorem 3.37]{StanMR3097424}, the random stopped sum $S_{\tau}\overset{\text{def}}{=}\sum\limits_{i=1}^{\tau}\eta_i$ of i.i.d. sub-exponential variables $(\eta_i)_i$ is again sub-exponential if $\mathbb{E}[(1+\delta)^{\tau}]<\infty$ for some $\delta>0$. Moreover,
$$\mathbb{P}[S_{\tau}>x]\sim\mathbb{E}[\tau]\mathbb{P}[\eta_1>x]\text{ as }x\rightarrow\infty.$$
Here, we are in a slightly different situation. In fact, by Lemma \ref{lem:cut the loop passing through a vertex into excursions}, we have that
$$\#(\mathcal{C}_{\alpha}(0,1)\cup\{0\})\overset{\text{law}}{=}\#\left(\bigcup\limits_{i=1}^{\xi(0,\mathcal{L}_{\alpha})}\mathrm{range}_i\right)$$
where $(\mathrm{range}_i)_i$ are i.i.d. variables which are independent of the Poisson loop soup and follow the distribution of the range of excursion under $\mathbb{P}^{ex}$. We see that
$$\max\limits_{i=1}^{\xi(0,\mathcal{L}_{\alpha})}\eta_i\leq\#\left(\bigcup\limits_{i=1}^{\xi(0,\mathcal{L}_{\alpha})}\mathrm{range}_i\right)\leq \sum\limits_{i=1}^{\xi(0,\mathcal{L}_{\alpha})}\eta_i$$
where $\eta_i\overset{\text{def}}{=}\#\mathrm{range}_i$ for all $i$. 
We will see that $\max\limits_{i=1}^{\xi(0,\mathcal{L}_{\alpha})}\eta_i$ and $\sum\limits_{i=1}^{\xi(0,\mathcal{L}_{\alpha})}\eta_i$ have the same tail behavior: 
On one hand, since $\xi(0,\mathcal{L}_{\alpha})$ has an exponentially decayed tail by Lemma \ref{lem:cut the loop passing through a vertex into excursions}, 
we can apply \cite[Theorem 3.37]{StanMR3097424}. Then, as $x\rightarrow\infty$,
$$\mathbb{P}\left[\sum\limits_{i=1}^{\xi(0,\mathcal{L}_{\alpha})}\eta_i>x\right]\sim\mathbb{E}[\xi(0,\mathcal{L}_{\alpha})]\mathbb{P}[\eta_1>x].$$
On the other hand,
\begin{align*}
 \mathbb{P}\left[\max\limits_{i=1}^{\xi(0,\mathcal{L}_{\alpha})}\eta_i>x\right]=&\mathbb{E}\left[1-\left(1-\mathbb{P}[\eta_1>x]\right)^{\xi(0,\mathcal{L}_{\alpha})}\right]\\
 =&\mathbb{P}[\eta_1>x]\mathbb{E}\left[\frac{1-\left(1-\mathbb{P}[\eta_1>x]\right)^{\xi(0,\mathcal{L}_{\alpha})}}{\mathbb{P}[\eta_1>x]\xi(0,\mathcal{L}_{\alpha})}\xi(0,\mathcal{L}_{\alpha})\right].
\end{align*}
By Lemma \ref{lem:tail of the length of first finite excursion}, $\mathbb{P}[\eta_1>x]\rightarrow 0$ as $x\rightarrow\infty$. 
Then, by the dominated convergence, 
\[
\mathbb{E}\left[\frac{1-\left(1-\mathbb{P}[\eta_1>x]\right)^{\xi(0,\mathcal{L}_{\alpha})}}{\mathbb{P}[\eta_1>x]\xi(0,\mathcal{L}_{\alpha})}\xi(0,\mathcal{L}_{\alpha})\right]
\overset{x\rightarrow\infty}{\sim}\mathbb{E}[\xi(0,\mathcal{L}_{\alpha})].
\]
Therefore, as $x\rightarrow\infty$,
$$\mathbb{P}\left[\max\limits_{i=1}^{\xi(0,\mathcal{L}_{\alpha})}\eta_i>x\right]\sim\mathbb{E}[\xi(0,\mathcal{L}_{\alpha})]\cdot \mathbb{P}[\eta_1>x].$$

Thus, we must have that as $x\rightarrow\infty$,
$$\mathbb{P}\left[\#\mathcal{C}_{\alpha}(0,1)>x\right]\sim\mathbb{E}[\xi(0,\mathcal{L}_{\alpha})]\mathbb{P}[\eta_1>x].$$
By Lemma \ref{lem:cut the loop passing through a vertex into excursions},
$$\mathbb{E}[\xi(0,\mathcal{L}_{\alpha})]=\alpha(G(0,0)-1).$$
By Lemma \ref{lem:tail of the length of first finite excursion}, as $x\rightarrow\infty$,
$$\mathbb{P}[\eta_1>x]\sim \frac{d^{d/2}}{(G(0,0)-1)(d/2-1)(2\pi G(0,0))^{d/2}}x^{1-\frac{d}{2}}.$$ 
Thus, as $x\rightarrow\infty$,
$$\mathbb{P}[\#\mathcal{C}_{\alpha}(0,1)>x]\sim \frac{\alpha d^{d/2}}{(d/2-1)(2\pi G(0,0))^{d/2}}x^{1-\frac{d}{2}}.$$ 

In particular, for $p\geq 0$,
$\mathbb{E}[(\#\mathcal{C}_{\alpha}(0,1))^{p}]<\infty$ iff $p<\frac{d}{2}-1$.
\end{proof}

\subsection{\texorpdfstring{$\thrcz\leq \thran$}{α\#≤α1}}
In this section we complete the proof of Theorem~\ref{thm:exponent of one-arm connectivity for d larger than 5} by showing the relation between $\thrcz$ and $\thran$:
\begin{prop}\label{prop: alpha * is smaller than alpha **}
 For $d\geq 5$, $\thrcz\leq\thran$.
\end{prop}
We prove Proposition~\ref{prop: alpha * is smaller than alpha **} by showing that if $\mathbb E[\#\mathcal C_\alpha(0)]<\infty$, 
then the probabilities of crossing annuli of large enough aspect ratio by chains of loops from $\mathcal L_\alpha$ are uniformly smaller than $1$. 
For a given annulus, we distinguish three possible situations: (a) one of the $5$ disjoint subannuli is crossed by a single loop from $\mathcal L_\alpha$, 
(b) there is a crossing with at least $3$ big loops, and (c) every crossing contains many loops. 
Probabilities of these events are estimated in $3$ lemmas below. 

\medskip

The first lemma estimates the probability that one of the $5$ disjoint subannuli of a given annulus is crossed by a loop. 
\begin{lem}\label{lem: five annuli}
For integers $\beta\geq 2$, $n\geq 1$, and $i=1,2,3,4,5$, let
\[ 
W_{n,i}=\{\exists \ell\in \mathcal{L}_{\alpha}~:~B(0,\beta^{i-1} n)\stackrel{\ell}\longleftrightarrow\partial B(0,\beta^{i} n)\}.
\]
Then, there exists $C(d)<\infty$ such that for $n\geq 1$,
\[
\mathbb{P}\left[\bigcup\limits_{i=1}^{5}W_{n,i}\right]\leq 5\alpha C(d)\beta^{2-d}.
\]
\end{lem}

\medskip

The next lemma estimates the probability that there is a chain of loops from a given vertex which contains at least $k$ loops of diameter $\geq m$. 
(In the proof of Proposition~\ref{prop: alpha * is smaller than alpha **} we will only need the case $k=3$.)
To state the lemma, we introduce some notation. 
For $x,y\in\mathbb{Z}^d$ and $m\geq 1$, define the events $J_{x,m}$ and $J_{x,y,m}$ by
\[
J_{x,y,m}=\{\exists \ell\in\mathcal{L}_{\alpha}:x,y\in \ell,Diam(\ell)\geq m\}\quad\mbox{and}\quad J_{x,m} = J_{x,x,m}.
\]
For $x\in\mathbb{Z}^d$, define the event
\[
E^{x}_{m,k}\overset{\text{def}}{=}\left\{
\begin{array}{r}
\bigcup\limits_{\substack{x_1,\ldots,x_k\\y_1,\ldots,y_{k-1}}}\{x\overset{\mathcal{L}_{\alpha}}{\longleftrightarrow}x_1\}\circ\{y_1\overset{\mathcal{L}_{\alpha}}{\longleftrightarrow}x_2\}\circ\cdots \circ\{y_{k-1}\overset{\mathcal{L}_{\alpha}}{\longleftrightarrow}x_k\}\\
\circ J_{x_1,y_1,m}\circ \cdots \circ J_{x_{k-1},y_{k-1},m}\circ J_{x_k,m}
\end{array}
\right\},
\]
 where the notation ``$\circ$'' means the disjoint occurrence in the sense of loops instead of edges. To be more precise, for $x,y\in\mathbb{Z}^d$ and a set of loops $O$, denote by $x\overset{O}{\longleftrightarrow}y$ the connection from $x$ to $y$ by loops in $O$:
\[
\{x\overset{O}{\longleftrightarrow}y\} = \left\{\exists \ell_1,\ldots,\ell_p\in O~:~x\in \ell_1,~\ell_1\cap \ell_2\neq \phi,\ldots,\ell_{p-1}\cap \ell_p\neq\phi,~y\in \ell_p\right\}
\]
with the conventions that $\{x\overset{\phi}{\longleftrightarrow}x\}$ is the sure event. 
Then, 
\[
E^{x}_{m,k}=\bigcup\limits_{\left\{\substack{\text{ disjoint finite loop sets }\\ O_1,\ldots,O_{k}\text{ such that}\\ 
x\overset{O_1}{\longleftrightarrow}x_1,~y_1\overset{O_2}{\longleftrightarrow}x_2,\ldots,y_{k-1}\overset{O_k}{\longleftrightarrow}x_k}\right\}}
\bigcup\limits_{\left\{\substack{\text{Different }\ell_1,\ldots,\ell_k\notin\bigcup\limits_{i=1}^{k}O_i\text{ such that}\\ 
Diam(\ell_1)\geq m,\ldots,Diam(\ell_k)\geq m\\ x_1\overset{\ell_1}{\longleftrightarrow}y_1,\ldots,x_{k-1}\overset{\ell_1}{\longleftrightarrow}y_{k-1},~x_k\in \ell_k}\right\}}
\left\{\bigcup\limits_{i=1}^{k}(O_{i}\cup\{\ell_i\})\subset\mathcal{L}_{\alpha}\right\}.
\]
Figure~\ref{fig:manybigloops} is an illustration of event $E^{x}_{m,3}$.
\begin{figure}[!htbp]
\begin{picture}(0,0)%
\includegraphics{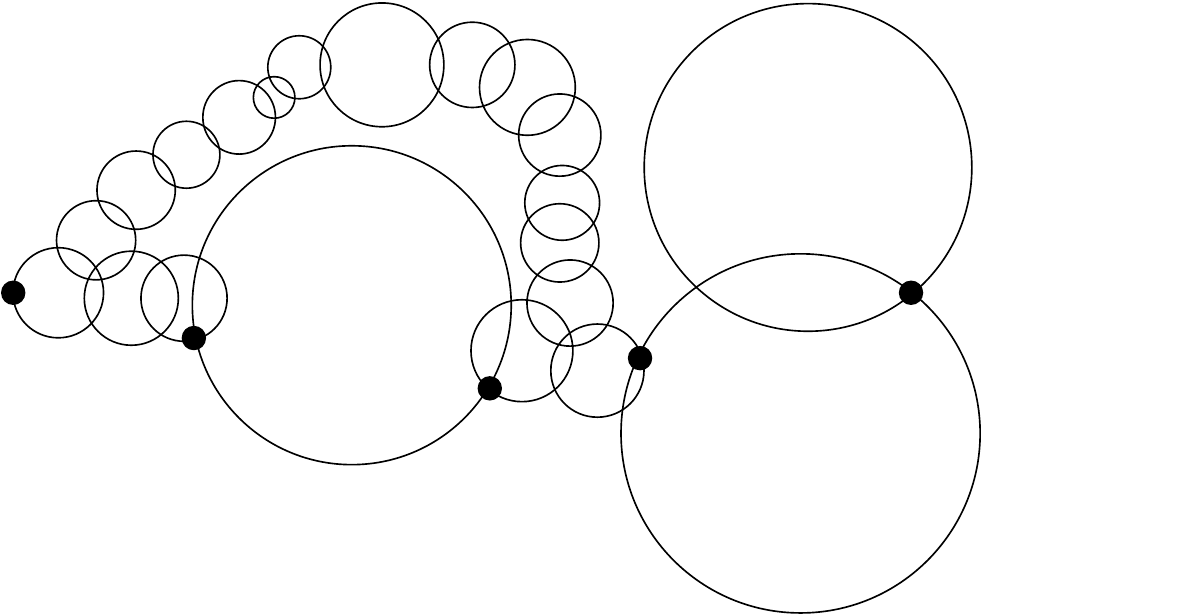}%
\end{picture}%
\setlength{\unitlength}{4144sp}%
\begingroup\makeatletter\ifx\SetFigFont\undefined%
\gdef\SetFigFont#1#2#3#4#5{%
  \reset@font\fontsize{#1}{#2pt}%
  \fontfamily{#3}\fontseries{#4}\fontshape{#5}%
  \selectfont}%
\fi\endgroup%
\begin{picture}(5445,2805)(1816,-1962)
\put(3241,-61){\makebox(0,0)[lb]{\smash{{\SetFigFont{12}{14.4}{\familydefault}{\mddefault}{\updefault}{\color[rgb]{0,0,0}$l_1$}%
}}}}
\put(5401,569){\makebox(0,0)[lb]{\smash{{\SetFigFont{12}{14.4}{\familydefault}{\mddefault}{\updefault}{\color[rgb]{0,0,0}$l_3$}%
}}}}
\put(5311,-1771){\makebox(0,0)[lb]{\smash{{\SetFigFont{12}{14.4}{\familydefault}{\mddefault}{\updefault}{\color[rgb]{0,0,0}$l_2$}%
}}}}
\put(2791,-781){\makebox(0,0)[lb]{\smash{{\SetFigFont{12}{14.4}{\familydefault}{\mddefault}{\updefault}{\color[rgb]{0,0,0}$x_1$}%
}}}}
\put(3781,-961){\makebox(0,0)[lb]{\smash{{\SetFigFont{12}{14.4}{\familydefault}{\mddefault}{\updefault}{\color[rgb]{0,0,0}$y_1$}%
}}}}
\put(4861,-871){\makebox(0,0)[lb]{\smash{{\SetFigFont{12}{14.4}{\familydefault}{\mddefault}{\updefault}{\color[rgb]{0,0,0}$x_2$}%
}}}}
\put(6121,-511){\makebox(0,0)[lb]{\smash{{\SetFigFont{12}{14.4}{\familydefault}{\mddefault}{\updefault}{\color[rgb]{0,0,0}$y_2=x_3$}%
}}}}
\put(1936,-556){\makebox(0,0)[lb]{\smash{{\SetFigFont{12}{14.4}{\familydefault}{\mddefault}{\updefault}{\color[rgb]{0,0,0}$x$}%
}}}}
\put(7246,614){\makebox(0,0)[lb]{\smash{{\SetFigFont{12}{14.4}{\familydefault}{\mddefault}{\updefault}{\color[rgb]{0,0,0}$Diam(l_1)\geq m$}%
}}}}
\put(7246,299){\makebox(0,0)[lb]{\smash{{\SetFigFont{12}{14.4}{\familydefault}{\mddefault}{\updefault}{\color[rgb]{0,0,0}$Diam(l_2)\geq m$}%
}}}}
\put(7246,-16){\makebox(0,0)[lb]{\smash{{\SetFigFont{12}{14.4}{\familydefault}{\mddefault}{\updefault}{\color[rgb]{0,0,0}$Diam(l_3)\geq m$}%
}}}}
\end{picture}%
\caption{Illustration of event $E^{x}_{m,k}$ for $k=3$}
\label{fig:manybigloops}
\end{figure}
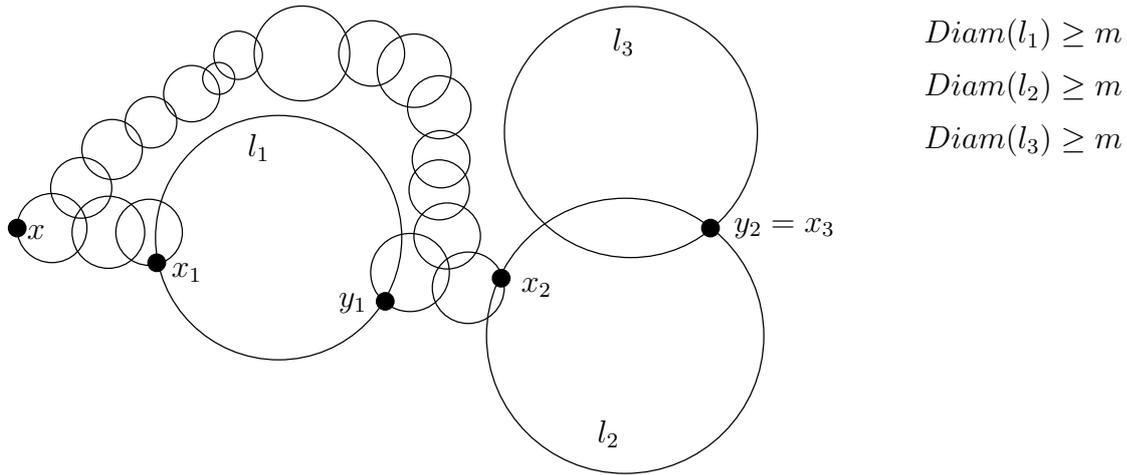
\begin{lem}\label{lem: several loops in order}
 For $d\geq 5$, $\alpha<\thrcz$, and $k\geq 1$, there exists constant $C_k(d,\alpha)$ such that 
\begin{equation}\label{eq:lslio1}
\mathbb{P}[E^0_{m,k}]\leq C_k(d,\alpha)\cdot m^{-(k(d-4)+2)}.
\end{equation}
\end{lem}

\medskip

Finally, the next lemma provides a bound on the probability of a long ``geodesic'' chain of loops from $0$:
\begin{lem}\label{lem: exponential decay of the total generation}
For $d\geq 5$ and $\alpha<\thrcz$, there exist $c(d,\alpha)>0$ and $C(d,\alpha)<\infty$ such that for $k\geq 0$
\[
\mathbb{P}[\#\mathcal{C}_{\alpha}(0,k)>0]\leq C(d,\alpha)\cdot e^{-c(d,\alpha)k}.
\]
\end{lem}

\bigskip

We can now deduce Proposition~\ref{prop: alpha * is smaller than alpha **} from the above three lemmas, and after that prove the lemmas.

\begin{proof}[Proof of Proposition \ref{prop: alpha * is smaller than alpha **}]
Let $d\geq 5$ and $\alpha<\thrcz$. 
To prove that $\alpha\leq \thran$, we need to show that there exist $\beta\geq 2$ and $r<1$ such that
\begin{equation}\label{eq:crossingr}
\limsup\limits_{n\rightarrow\infty}\mathbb{P}[B(0,n)\overset{\mathcal{L}_{\alpha}}{\longleftrightarrow}\partial B(0,\beta^5 n)]\leq r.
\end{equation}
We cut the annulus into $5$ concentric annuli and denote by $W_{n,i}$ the one loop crossing event for each annulus for $i=1,2,3,4,5$:
\[
W_{n,i}=\left\{\exists \ell\in \mathcal{L}_{\alpha}~:~B(0,\beta^{i-1} n)\stackrel{\ell}\longleftrightarrow\partial B(0,\beta^{i} n)\right\}.
\]
By Lemma \ref{lem: five annuli}, we can choose $\beta$ large enough such that for $n\geq 1$,
\[
\mathbb{P}\left[\bigcup\limits_{i=1}^{5}W_{n,i}\right]\leq \frac{1}{2}.
\]

Next, we estimate the probability that a path of loops from $B(0,n)$ to $\partial B(\beta^5n)$ contains at least $3$ large loops. 
Take $m=\lfloor n^{5/6}\rfloor$. By Lemma~\ref{lem: several loops in order}, 
\[
\mathbb P\left[\bigcup_{x\in\partial B(0,n)} E^x_{m,3}\right] \leq C(d,\alpha)\cdot \#\partial B(0,n)\cdot m^{10-3d} ,
\]
which tends to $0$ as $n\to\infty$. 

On the event 
\[
\left\{B(0,n)\overset{\mathcal{L}_{\alpha}}{\longleftrightarrow}\partial B(0,\beta^5 n)\right\}\cap\left( \bigcup\limits_{i=1}^{5}W_{n,i}\right)^c\cap\left(\bigcup_{x\in\partial B(0,n)} E^x_{m,3}\right)^c,
\]
every path of loops in $\mathcal L_\alpha$ from $B(0,n)$ to $\partial B(0,\beta^5n)$ must cross at least one of the $5$ subannuli only by loops of diameter smaller than $m$. 
Since the infinity distance between the inner boundary and the outer boundary of each subannulus is at least $(\beta-1)n$, 
every such path must consist of at least $\frac{(\beta-1)n}{m}$ loops, which implies that for some $x\in\partial B(0,n)$, $\mathcal{C}_{\alpha}(x,\frac{(\beta-1)n}{m})\neq\phi$. 
By Lemma~\ref{lem: exponential decay of the total generation}, the probability of this event is bounded from above by $C(d,\alpha)\cdot n^{d-1}\exp\{-c(d,\alpha)n^{1/6}\}$, 
which goes to $0$ as $n\to\infty$.

Finally, we conclude \eqref{eq:crossingr} with $r=1/2$.
\end{proof}

\medskip

It remains to prove the lemmas.

\begin{proof}[Proof of Lemma \ref{lem: five annuli}]
 By Lemmas \ref{lem:one loop connection on Zd} and \ref{lem: capacity}, there exists $C(d)<\infty$ such that for $n\geq 1$,
\begin{align*}
\mathbb{P}\left[\bigcup\limits_{i=1}^{5}W_{n,i}\right]\leq 
&\sum\limits_{i=1}^{5}\mathbb{P}\left[\exists \ell\in\mathcal{L}_{\alpha}~:~ B(0,\beta^{i-1} n)\stackrel{\ell}\longleftrightarrow\partial B(0,\beta^{i} n)\right]\\
\leq&\sum\limits_{i=1}^{5}\alpha\mu(\ell~:~ B(0,\beta^{i-1} n)\stackrel{\ell}\longleftrightarrow\partial B(0,\beta^{i} n))\\
  \leq &5\alpha C(d)\beta^{2-d}.\qedhere
 \end{align*}
\end{proof}

\medskip

\begin{proof}[Proof of Lemma \ref{lem: several loops in order}]\ 
We prove \eqref{eq:lslio1} by induction in the following three steps:
\begin{itemize}
\item[I)] Proof of \eqref{eq:lslio1} for $k=1$: By considering
 $$\min\{q\geq 0:\text{there exists }\ell\in\mathcal{L}_{\alpha}\text{ such that }\ell\cap\mathcal{C}_{\alpha}(0,q)\neq\phi,Diam(\ell)\geq m\}$$
 and then using the first moment method,
\begin{align*}
  \mathbb{P}[E^{0}_{m,1}]\leq & \mathbb{E}[\#\mathcal{C}_{\alpha}(0)]\cdot \mathbb{P}[\exists \ell\in\mathcal{L}_{\alpha}~:~0\in \ell,~Diam(\ell)\geq m]\\
  \leq&\mathbb{E}[\#\mathcal{C}_{\alpha}(0)]\cdot \mathbb{P}[\exists \ell\in\mathcal{L}_{\alpha}~:~0\stackrel{\ell}\longleftrightarrow\partial B(0,\lfloor m/2\rfloor)]\\
  \leq&\alpha\cdot \mathbb{E}[\#\mathcal{C}_{\alpha}(0)]\cdot \mu(\ell~:~0\stackrel{\ell}\longleftrightarrow\partial B(0,\lfloor m/2\rfloor)).
 \end{align*}
 The proof of \eqref{eq:lslio1} for $k=1$ is complete by Lemma \ref{lem:one loop connection on Zd}.
 \end{itemize}
 \item[II)] Suppose \eqref{eq:lslio1} holds for $k\leq K$. 
Then for $k=K+1$, since $E^x_{m,K}$ are increasing events and $\mathcal{L}_{\alpha}$ is translation invariant, by the first moment method,
 \begin{equation}\label{eq:lslio2}
  \mathbb{P}[E^{0}_{m,K+1}]\leq\sum\limits_{x,y\in\mathbb{Z}^d}\mathbb{P}\left[\{0\overset{\mathcal{L}_{\alpha}}{\longleftrightarrow} x\}\circ\{\exists \ell\in\mathcal{L}_{\alpha}:x,y\in \ell,Diam(\ell)\geq m\}\circ E^y_{m,K}\right].
 \end{equation}
 where $A\circ B$ means the disjoint occurrence for two increasing events $A$ and $B$ in the sense of loops instead of edges. 
We set $\omega(\ell)=1$ iff $\ell\in\mathcal{L}_{\alpha}$. Similarly to the primitive loops considered in \cite[Section~2.1]{LemaireLeJan}, 
the distribution of the random configuration measure $(\omega(\ell))_{\ell}$ is a product measure by the definition of Poisson random measure. 
It is known that the BK inequality holds for product measure on $\{0,1\}^m$ for finite $m\geq 1$, see e.g. \cite[Theorem 2.12]{GrimmettMR1707339}. 
Although $\{0,1\}^{\{\text{loop space on }\mathbb{Z}^d\}}$ is not finite, the finite volume approximation works in many situations. 
Indeed, the increasing event considered in this proof can be approximated by monotone sequence of finitely loop dependent events. Thus, we can apply BK inequality:
 \begin{equation*}
  \eqref{eq:lslio2}\leq\sum\limits_{x,y\in\mathbb{Z}^d}\mathbb{P}[0\overset{\mathcal{L}_{\alpha}}{\longleftrightarrow} x]\cdot 
\mathbb{P}[\exists \ell\in\mathcal{L}_{\alpha}:x,y\in \ell,Diam(\ell)\geq m]\cdot \mathbb{P}[E^y_{m,K}].
 \end{equation*}
By translation invariance of $\mathcal{L}_{\alpha}$, 
\begin{equation*}
\mathbb{P}[E^{0}_{m,K+1}]\leq\mathbb{E}[\#\mathcal{C}_{\alpha}(0)]\cdot \mathbb{P}[E^{0}_{m,K}]\cdot \sum\limits_{y\in\mathbb{Z}^d}\mathbb{P}[\exists \ell\in\mathcal{L}_{\alpha}~:~0,y\in \ell,Diam(\ell)\geq m].
\end{equation*}
 In the next step, we will bound $\sum\limits_{y\in\mathbb{Z}^d}\mathbb{P}[\exists \ell\in\mathcal{L}_{\alpha}~:~0,y\in \ell,~Diam(\ell)\geq m]$ by $C(d,\alpha)\cdot m^{4-d}$, which 
will finish the proof of the lemma.
 \item[III)] 
By comparing $\#\mathcal{C}_{\alpha}(0,1)$ with $m^2$, 
 \begin{align*}
  \sum\limits_{y\in\mathbb{Z}^d}\mathbb{P}[\exists \ell\in\mathcal{L}_{\alpha}~:~0,y\in \ell,~Diam(\ell)\geq m]
\leq &\mathbb{E}\left[1_{\{\#\mathcal{C}_{\alpha}(0,1)\geq m^2\}}\cdot\#\mathcal{C}_{\alpha}(0,1)\right]\\
  &+m^2\mathbb{P}[\exists \ell\in\mathcal{L}_{\alpha}~:~0\in \ell,~Diam(\ell)\geq m].
 \end{align*}
 By Lemma \ref{lem:tail of the range of the loop intersecting 0}, there exists $C(d,\alpha)<\infty$ such that 
\[
\mathbb{E}\left[1_{\{\#\mathcal{C}_{\alpha}(0,1)\geq m^2\}}\cdot\#\mathcal{C}_{\alpha}(0,1)\right]\leq C(d,\alpha)\cdot m^{4-d}.
\]
The second term is estimated by $m^2\cdot \mathbb P[E^0_{m,1}]\leq C'(d,\alpha)\cdot m^{4-d}$. 
The proof of the lemma is thus complete. 
\end{proof}

\medskip

\begin{proof}[Proof of Lemma \ref{lem: exponential decay of the total generation}]
Recall the definition of $\mathcal U_\alpha(0,K)$ from \eqref{def:U0K}. 

For $d\geq 5$ and $\alpha<\thrcz$, there exists $K(d,\alpha)\in\mathbb{N}$ large enough such that 
\[
\mathbb{E}[\#\mathcal{U}_{\alpha}(0,K(d,\alpha))]<e^{-1}.
\]
Since $\left(\#\mathcal{C}(0,K(d,\alpha)i)\right)_{i\geq 0}$ is dominated by a sub-critical Galton-Watson process with offspring distribution 
$\mathbb{P}[\#\mathcal{U}_{\alpha}(0,K(d,\alpha))\in\cdot]$,
\[
\mathbb{P}[\#\mathcal{C}_{\alpha}(0,k)>0]\leq \mathbb{P}\left[\#\mathcal{C}_{\alpha}\left(0,K(d,\alpha)\left\lfloor \frac{k}{K(d,\alpha)}\right\rfloor\right)>0\right]
\leq \exp\left\{-\left\lfloor \frac{k}{K(d,\alpha)}\right\rfloor\right\}.\qedhere
\]
\end{proof}

\subsection{Asymptotic expression for \texorpdfstring{$\thr$}{α1} as \texorpdfstring{$d\to\infty$}{d→∞}}

The main result of this section is Proposition~\ref{prop: asymptotic of critical value in high dimensions}, 
which shows that all the critical thresholds defined in this paper asymptotically coincide as the dimension $d\to\infty$, 
and gives also their asymptotic value. 
The proof involves a careful estimate of $\mathbb{E}[\#\mathcal{C}_{\alpha}(0,1)]$ together with an upper bound of critical value $\thr$ in \cite[Proposition 4.3(ii)]{LemaireLeJan}.
\begin{prop}\label{prop: asymptotic of critical value in high dimensions}
Asymptotically, as $d\to\infty$,
\[
2d-6+O(d^{-1})~\leq ~\thrcz~\leq~ \thr~\leq~ 2d+\frac{3}{2}+O(d^{-1}).
\]
\end{prop}
\begin{proof}
The upper bound of the critical value $\thr$ follows from the comparison between the loop percolation 
and Bernoulli bond percolation in \cite[Proposition 4.3(ii)]{LemaireLeJan}, 
\[
\left(1 - \frac{1}{(2d)^2}\right)^{\thr} \leq 1 - p_c,
\]
where $p_c$ is the critical value of Bernoulli bond percolation, 
and the asymptotic expansion for $p_c$ as in \cite[(11.19)]{SladeMR2239599}, 
\[
p_c = \frac{1}{2d} + \frac{1}{(2d)^2} + O(d^{-3}).
\]

For the lower bound on $\thrcz$, recall from the proof of Proposition~\ref{prop:posivity of alpha *} that for $d\geq 5$, 
\begin{align*}
\thrcz\geq &\inf\{\alpha>0~:~\mathbb E \#\mathcal C_\alpha(0,1)<1\}\\
=&\inf\left\{\alpha>0:\sum\limits_{x\in\mathbb{Z}^d,x\neq 0}1-\left(1-\left(\frac{G(0,x)}{G(0,0)}\right)^2\right)^{\alpha}<1\right\}.
\end{align*}
For $\alpha\geq 1$, 
\[
\sum\limits_{x\in\mathbb{Z}^d,x\neq 0}1-\left(1-\left(\frac{G(0,x)}{G(0,0)}\right)^2\right)^{\alpha}\leq \alpha\sum\limits_{x\in\mathbb{Z}^d,x\neq 0}\left(\frac{G(0,x)}{G(0,0)}\right)^2.
\] 
Thus, the lower bound on $\thrcz$ follows from the following claim: 
\begin{equation}\label{eq:paocvihd2}
\sum\limits_{x\in\mathbb{Z}^d,x\neq 0}\left(\frac{G(0,x)}{G(0,0)}\right)^2=\frac{1}{2d}\left(1+\frac{3}{d}+O(d^{-2})\right),\text{ as }d\rightarrow\infty.
\end{equation}
For \eqref{eq:paocvihd2}, it suffices to show that
\begin{equation}\label{eq:Gexpansion}
\sum\limits_{x\in\mathbb{Z}^d}(G(0,x))^2=1+\frac{3}{2d}+\frac{15}{4d^2}+O(d^{-3})\quad\mbox{and}\quad
G(0,0)= 1+\frac{1}{2d}+\frac{3}{4d^2}+O(d^{-3}).
\end{equation}
We use Fourier transforms.  
For an absolutely summable function $f~:~\mathbb Z^d\to \mathbb C$ and $k\in]-\pi,\pi[^d$, 
the Fourier transform of $f$ at $k$ is defined by 
\[
\hat f(k) = \sum_{x\in\mathbb Z^d} e^{ik\cdot x} f(x), \quad\mbox{where } k\cdot x = \sum_{j=1}^d k_jx_j.
\]
Let $D(x) = \frac{1}{2d} \mathrm{1}_{\{||x||_2 = 1\}}$ be the transition probability for SRW from $0$. 
Then $\hat{D}(k)=\frac{1}{d}\sum\limits_{i=1}^{d}\cos(k_i)$ and $\hat{G}(k)=\frac{1}{1-\hat{D}(k)}$ in the sense of $L^2(\mathrm{d}k)$ for $d\geq 5$.

We first consider $\sum\limits_{x\in\mathbb{Z}^d}(G(0,x))^2$.
By Parseval's identity, for $d\geq 5$, 
\[
\sum\limits_{x\in\mathbb{Z}^d}(G(0,x))^2=\frac{1}{(2\pi)^d}\int\limits_{]-\pi,\pi[^d}\left(\frac{1}{1-\hat{D}(k)}\right)^2\mathrm{d}^dk.
\]
It will be convenient to use the probabilistic interpretation. 
Let $(U_i)_i$ be independent random variables uniformly distributed in $]-\pi,\pi[$, and define $Z_d=\frac{1}{d}\sum\limits_{i=1}^{d}\cos(U_i)$. Then,
\[
\sum\limits_{x\in\mathbb{Z}^d}(G(0,x))^2=\mathbb{E}[(1-Z_d)^{-2}].
\]
We expand $(1-Z_d)^{-2}=\sum\limits_{n\geq 0}(n+1)Z_d^n$ into a power series of $Z_d$.
We first estimate the error term 
\[
R_{d,2m}\overset{\mathrm{def}}{=}\left|(1-Z_d)^{-2}-\sum\limits_{n=0}^{2m-1}(n+1)Z_d^n\right|.
\]
If $|Z_d|<1/2$,
\[
R_{d,2m}\leq Z_d^{2m}\cdot\sum\limits_{j=0}^{\infty}(2m+1+j)|Z_d|^{j}\leq 8mZ_d^{2m},
\]
if $|Z_d|\geq 1/2$,
\[
R_{d,2m}\leq (1-Z_d)^{-2}+\left|\sum\limits_{n=0}^{2m-1}(n+1)Z_d^n\right|
\leq (1 - Z_d)^{-2}1_{Z_d\geq 1/2} + 1 + m(2m-1).
\]
Note that 
\[
\mathbb E [Z_d^2] = \frac{1}{2d}\quad \mbox{and}\quad \mathbb E [Z_d^{2m}] \leq \frac{C(m)}{d^m}, \quad\mbox{for }m\geq 2,
\]
and by the exponential Markov inequality, $\mathbb P[|Z_d|\geq \frac 12] \leq 2\cdot e^{-d/8}$. 
Thus, for some $C'(m)<\infty$, 
\[
\mathbb E [R_{d,2m}] \leq \frac{C'(m)}{d^m} + \mathbb E[(1 - Z_d)^{-2}1_{Z_d\geq 1/2}].
\]
By H\"older's inequality and the exponential Markov inequality, 
\begin{equation}\label{eq:paocvihd3}
\mathbb{E}[(1-Z_d)^{-2}1_{Z_d\geq 1/2}]
 \leq (\mathbb{E}[(1-Z_d)^{-7/3}])^{6/7}\cdot e^{-d/56}.
\end{equation}
We will show that there exists a universal constant $C<\infty$ such that for $d\geq 5$, 
\begin{equation}\label{eq:paocvihd4}
\mathbb{E}[(1-Z_d)^{-7/3}]<C.
\end{equation}
Once \eqref{eq:paocvihd4} is proved, by using the bound on $\mathbb E[R_{d,6}]$, we get
\[
\sum\limits_{x\in\mathbb{Z}^d}(G(0,x))^2=
\sum\limits_{n=0}^5(n+1)\mathbb E[Z_d^n] + O(d^{-3}),
\]
and the first part of \eqref{eq:Gexpansion} follows by direct calculation of moments of $Z_d$ up to order $O(d^{-3})$. 

It remains to prove \eqref{eq:paocvihd4}. 
By convexity of the function $h\mapsto (1-h)^{-7/3}$ and the definition of $\hat D(k)$,
\begin{equation}\label{eq:Zdrecursion}
\mathbb{E}[(1-Z_d)^{-7/3}]\leq\frac 5d\mathbb E[(1 - Z_5)^{-7/3}] + \frac{d-5}{d}\mathbb E[(1-Z_{d-5})^{-7/3}]. 
\end{equation}
Thus, the uniform bound in \eqref{eq:paocvihd4} follows by induction from \eqref{eq:Zdrecursion} as soon as we show that 
for any $d\geq 5$, $\mathbb{E}[(1-Z_d)^{-7/3}]<\infty$. 
This follows from the definition of $\hat D(k)$, the fact that $\min\limits_{x\in]-\pi,\pi[}\frac{1-\cos(x)}{x^2}>0$, 
and the finiteness of the integral $\frac{1}{(2\pi)^d}\int\limits_{]-\pi,\pi[^d}\frac{1}{||k||_2^{14/3}}\mathrm{d}^dk$ 
for any $d\geq 5$. 

The proof of the first expansion in \eqref{eq:Gexpansion} is complete. 
Since $\frac{1}{1-\hat{D}(k)}\in L^1$ by \eqref{eq:paocvihd4}, 
the expansion for $G(0,0)$ in \eqref{eq:Gexpansion} can be done similarly by the inverse Fourier transform. 
We omit the details, and complete the proof. 
\end{proof}

\section{Refined lower bound for \texorpdfstring{$d=3,4$}{d=3,4}}\label{sec: dimension three or four}
\subsection{\texorpdfstring{$d=3$}{d=3}: proof of Theorem~\ref{thm: d-2 is not the right exponent for d=3}}
In this section we prove that $n^{2-d}$ is not the correct order of decay of the one arm probability in dimension $d=3$ 
by providing a lower bound on the one arm probability of order $n^{2-d+\varepsilon}$. 
The key ingredient for the proof is the following lemma which gives a lower bound of the expected capacity of the open cluster at $0$ 
formed by the loops $(\mathcal{L}_{\alpha})^{B(0,k)}$ contained inside the box $B(0,k)$.
\begin{lem}\label{lem:lower bound box expected capacity}
For $d=3$ and $\alpha>0$, denote by $\mathcal{C}_{\alpha}^k(0)$ the open cluster at $0$ formed by the loops $(\mathcal{L}_{\alpha})^{B(0,k)}$ 
contained inside the box $B(0,k)$. Then, there exist positive constants $\epsilon(\alpha)$ and $c(\alpha)$ such that 
\[
\mathbb{E}[\Capa(\mathcal{C}_{\alpha}^k(0))]\geq c(\alpha)\cdot k^{\epsilon(\alpha)}.
\]
\end{lem}
Before proving the lemma, we show how to deduce Theorem~\ref{thm: d-2 is not the right exponent for d=3} from it. 
\begin{proof}[Proof of Theorem~\ref{thm: d-2 is not the right exponent for d=3}]
Let $d =3$ and $\alpha>0$. Take $\mathcal C_\alpha^{\lfloor n/2\rfloor }(0)$ as in the statement of Lemma~\ref{lem:lower bound box expected capacity}.
We always have
\[
\mathbb{P}[0\overset{\mathcal{L}_{\alpha}}{\longleftrightarrow}\partial B(0,n)]\geq
\mathbb{P}[\exists \ell\in\mathcal{L}_{\alpha}~:~\mathcal{C}_{\alpha}^{\lfloor n/2\rfloor}(0)\stackrel{\ell}\longleftrightarrow \partial B(0,n)].
\]
Note that $\mathcal{C}_{\alpha}^{\lfloor n/2\rfloor }(0)$ depends only on the loops $(\mathcal{L}_{\alpha})^{B(0,\lfloor n/2\rfloor)}$ inside the box $B(0,\lfloor n/2\rfloor)$. 
Thus, it is independent from the loops $(\mathcal{L}_{\alpha})_{(B(0,\lfloor n/2\rfloor))^c}$ intersecting $(B(0,\lfloor n/2\rfloor))^c$. 
Then, by Lemma \ref{lem:one loop connection on Zd}, there exists $c>0$ such that
\begin{multline}\label{eq:dintrefd1}
\mathbb{P}[\exists \ell\in\mathcal{L}_{\alpha}~:~\mathcal{C}_{\alpha}^{\lfloor n/2\rfloor }(0)\stackrel{\ell}\longleftrightarrow \partial B(0,n)]
\geq \mathbb{E}\left[1-\exp\{-\alpha\mu(\ell~:~\mathcal{C}_{\alpha}^{\lfloor n/2\rfloor }(0)\stackrel{\ell}\longleftrightarrow\partial B(0,n))\}\right]\\
\geq \mathbb{E}\left[1-\exp\{-\alpha c n^{-1}\Capa(\mathcal{C}_{\alpha}^{\lfloor n/2\rfloor }(0))\}\right].
 \end{multline}
 Since by Lemma~\ref{lem: capacity}, $n^{-1}\Capa\left(\mathcal{C}_{\alpha}^{\lfloor n/2\rfloor }(0)\right)\leq n^{-1}\Capa(B(0,\lfloor n/2\rfloor))$ is uniformly bounded, 
there exists a constant $c = c(\alpha)$ such that
 \begin{equation*}
 \eqref{eq:dintrefd1}\geq c\cdot n^{-1}\cdot\mathbb{E}[\Capa(\mathcal{C}_{\alpha}^{\lfloor n/2\rfloor }(0))].
 \end{equation*}
The proof is complete by Lemma \ref{lem:lower bound box expected capacity}.
\end{proof}

\medskip

We complete this section with the proof of the remaining lemma.
\begin{proof}[Proof of Lemma~\ref{lem:lower bound box expected capacity}]
Let $d = 3$ and $\alpha>0$.  
\begin{itemize}
\item[I)] By Lemma~\ref{lem:one loop connection on Zd}, there exist constants $\lambda_1>1$ and $c>0$ such that for $n\geq 1$, $m\geq 2n$, $M\geq\lambda_1m$, and $K\subset B(0,n)$, 
\begin{multline}\label{eq:llbbec1}
\mu(\ell~:~K\stackrel{\ell}\longleftrightarrow\partial B(0,m),~\ell\subset B(0,M),~\Capa(\ell)>cm)\\
\geq \mu(\ell~:~K\stackrel{\ell}\longleftrightarrow\partial B(0,m),~\Capa(\ell)>cm) - 
\mu(\ell~:~K\stackrel{\ell}\longleftrightarrow\partial B(0,M))\\
\geq c\cdot \Capa(K)\cdot m^{-1}.
 \end{multline}
 \item[II)] Denote by $\mathcal{C}_{\alpha}(K,m,M)$ the set of vertices visited by the loops from $\mathcal L_\alpha$ which are contained in $B(0,M)$ 
and intersect both $K$ and $\partial B(0,m)$.

We claim that there exists $\lambda_2 = \lambda_2(\alpha)>1$ such that 
for $n\geq 1$, $m\geq 2n$, $M>\lambda_2m$, and $K\subset B(0,n)$,
\begin{equation}\label{eq:llbbec2}
\mathbb{E}[\Capa(\mathcal{C}_{\alpha}(K,m,M))]\geq 2\cdot \Capa(K).
\end{equation}
Indeed, for the constant $c$ as in I), 
\begin{multline}\label{eq:llbbec3}
\mathbb{E}[\Capa(\mathcal{C}_{\alpha}(K,m,M)]=c\cdot \int\limits_{0}^{\infty}\mathrm{d}p\cdot \mathbb{P}[\Capa(\mathcal{C}_{\alpha}(K,m,M))>cp]\\
\geq c\cdot\int\limits_{m}^{\lfloor M/\lambda_1\rfloor}\mathrm{d}p\cdot\mathbb{P}
\left[\exists \ell\in\mathcal L_\alpha~:~K\stackrel{\ell}\longleftrightarrow \partial B(0,\lceil p\rceil),~\ell\subset B(0,M),~\Capa(\ell)> cp\right]\\
\stackrel{\eqref{eq:llbbec1}}\geq c\cdot\int\limits_{m}^{\lfloor M/\lambda_1\rfloor}\mathrm{d}p\cdot\left(
1 - \exp\left\{-\alpha\cdot c\cdot \Capa(K)\cdot \frac{1}{\lceil p\rceil}\right\}
\right)
 \end{multline}
Since $\Capa(K)\cdot\frac{1}{\lceil p\rceil}\leq \Capa(B(0,\lceil p\rceil))\cdot \frac{1}{\lceil p\rceil}$ is bounded, see Lemma~\ref{lem: capacity}, 
there exists $c' = c'(\alpha)>0$ such that 
\begin{equation}\label{eq:llbbec4}
\eqref{eq:llbbec3}\geq 
c\cdot c'\cdot \Capa(K)\cdot \int\limits_{m}^{M/\lambda_1}\frac{\mathrm{d}p}{p} = 
c\cdot c'\cdot \Capa(K)\cdot\log\left(\frac{M}{\lambda_1m}\right).
\end{equation}
By choosing $\lambda_2 = \lambda_2(\alpha)>1$ such that $c\cdot c'\cdot \log\left(\frac{\lambda_2}{\lambda_1}\right)> 2$, we get \eqref{eq:llbbec2}. 
 \item[III)] We can now complete the proof of Lemma~\ref{lem:lower bound box expected capacity} by iterating \eqref{eq:llbbec2}.
Take the deterministic sequence 
\[
M_0 = 1, \quad \text{and}\quad M_{i+1}=\lceil 2\lambda_2 M_{i}\rceil\quad \text{for $i\geq 0$},
\]
and the sequence of random subsets of $\mathbb Z^d$
\[
C_0=\{0\}, \quad\text{and}\quad C_{i+1}=\mathcal{C}_{\alpha}(C_{i},2M_{i},M_{i+1})\quad\text{for $i\geq 0$}.
\]
Note that for all $i$, $C_i\subset B(0,M_i)$, and the sets of loops forming $C_{i}$'s are disjoint for different $i$'s. Thus, by \eqref{eq:llbbec2}, for $i\geq 1$, 
\begin{equation}\label{eq:llbbec5}
\mathbb{E}[\Capa(C_i)|C_{i-1}]\geq 2\cdot \Capa(C_{i-1}),
\end{equation}
and by iteration of \eqref{eq:llbbec5}, 
\begin{equation}\label{eq:llbbec6}
\mathbb E[\Capa(C_i)] \geq 2^i\cdot \Capa(\{0\}). 
\end{equation}
For $k\geq 1$, let $T_k = \max\{i\geq 0~:~M_i\leq k\}$. Note that $C_{T_k}\subset \mathcal C_\alpha^k(0)$ and $T_k\geq \frac{\log k}{\log\lceil 2\lambda_2\rceil} -1$. 
By the monotonicity of capacity and \eqref{eq:llbbec6}, there exist $c = c(\alpha)$ and $\epsilon = \epsilon(\alpha)$ such that 
\[
\mathbb{E}[\Capa(\mathcal{C}_{\alpha}^k(0))]\geq \mathbb E[\Capa(C_{T_k})]\geq 2^{T_k}\cdot \Capa(\{0\})\geq c\cdot n^{\epsilon}.\qedhere
\]
\end{itemize}
\end{proof}

\subsection{\texorpdfstring{$d=4$}{d=4}: proof of Theorem~\ref{thm: refined lower bound for d=4}}
In this section we prove a lower bound on the one arm probability in dimension $d=4$, which is of order $n^{2-d}(\log n)^{\epsilon}$. 
This is better than the bound by the probability of single big loop obtained in Theorem~\ref{thm:at most polynomial decay of the connectivity}.

Since the proof of this fact is very similar to the proof of Theorem~\ref{thm: d-2 is not the right exponent for d=3}, 
we will only discuss necessary modifications. 
The role of Lemma~\ref{lem:lower bound box expected capacity} used in the proof of Theorem~\ref{thm: d-2 is not the right exponent for d=3}, 
is now played by the following lemma. 
\begin{lem}\label{lem:big box expected capacity d=4}
For $d=4$ and $\alpha>0$, let $\mathcal{C}_{\alpha}^k(0)$ be the open cluster at $0$ formed by the loops $(\mathcal{L}_{\alpha})^{B(0,k)}$ 
contained inside the box $B(0,k)$. Then, there exist positive constants $\epsilon = \epsilon(\alpha)$ and $c = c(\alpha)$ such that 
\[
\mathbb{E}[\Capa(\mathcal{C}_{\alpha}^k(0))]\geq c\cdot (\log n)^{\epsilon}.
\]
\end{lem}
The proof of Lemma~\ref{lem:big box expected capacity d=4} consists of the same $3$ steps as the proof of Lemma~\ref{lem:lower bound box expected capacity}:
\begin{itemize}
\item[I)]
By Lemma~\ref{lem:one loop connection on Zd}, there exist constants $\lambda_1>1$ and $c>0$ such that for $n\geq 1$, $m\geq 2n$, $M\geq\lambda_1m$, and $K\subset B(0,n)$, 
\begin{equation}\label{eq:capd4}
\mu\left(\ell~:~K\stackrel{\ell}\longleftrightarrow\partial B(0,m),~\ell\subset B(0,M),~\Capa(\ell)>c\cdot \frac{m^2}{\log m}\right)
\geq c\cdot \Capa(K)\cdot m^{-2}.
\end{equation}
\item[II)]
Denote by $\mathcal{C}_{\alpha}(K,m,M)$ the set of vertices visited by the loops from $\mathcal L_\alpha$ which are contained in $B(0,M)$ 
and intersect both $K$ and $\partial B(0,m)$. Then using \eqref{eq:capd4}, one can show similarly to \eqref{eq:llbbec3} and \eqref{eq:llbbec4} that 
there exists $c = c(\alpha)$ such that 
\[
\mathbb{E}[\Capa(\mathcal{C}_{\alpha}(K,m,M)] \geq c\cdot \Capa(K)\cdot \log\left(\frac{\log M}{\log m}\right).
\]
In particular, there exists $\lambda_2 = \lambda_2(\alpha)>1$ such that for all $M\geq m^{\lambda_2}$, 
\[
\mathbb{E}[\Capa(\mathcal{C}_{\alpha}(K,m,M)]\geq 2\cdot \Capa(K),
\]
which is an analogue of \eqref{eq:llbbec2}.
\item[III)]
Take the deterministic sequence 
\[
M_0 = 1, \quad \text{and}\quad M_{i+1}=\lceil (2M_{i})^{\lambda_2}\rceil\quad \text{for $i\geq 0$},
\]
and the sequence of random subsets of $\mathbb Z^d$
\[
C_0=\{0\}, \quad\text{and}\quad C_{i+1}=\mathcal{C}_{\alpha}(C_{i},2M_{i},M_{i+1})\quad\text{for $i\geq 0$}.
\]
For $k\geq 1$, let $T_k = \max\{i\geq 0~:~M_i\leq k\}$. Then $T_k$ is of order $\log\log k$, and 
$\mathbb{E}[\Capa(\mathcal{C}_{\alpha}^k(0))]\geq \mathbb E[\Capa(C_{T_k})]\geq c\cdot(\log k)^{\epsilon}$, 
for some $c = c(\alpha)$ and $\epsilon = \epsilon(\alpha)$. \qed
\end{itemize}

\section{Loop percolation on general graphs}\label{sec: general graphs}

In this section we discuss several properties of the loop percolation on a general graph:
\begin{itemize}\itemsep0pt
\item 
triviality of the tail sigma-algebra, see Proposition~\ref{prop:zero-one law for tail event},
\item
$\kappa_c(\alpha)\leq 0$ for any recurrent graph, see Proposition~\ref{prop: existence of percolation for recurrent graph and zero kappa},
\item
continuity of $\kappa_c(\alpha)$, see Proposition~\ref{Continuity of kappa_c(alpha)}.
\end{itemize}

We begin with a $0-1$ law for tail events. 
\begin{prop}\label{prop:zero-one law for tail event}
For any $A\in\bigcap\limits_{K\text{ finite}}\mathcal{F}_{K^c}$, $\mathbb P[A]\in \{0,1\}$.
\end{prop}
\begin{proof}
Let $(K_n)_n$ be an increasing sequence of finite subsets of vertices in $V$ such that $\bigcup\limits_{n} K_n = V$.
As $1_A\in L^p(\mathbb{P})$,
$$\mathbb{P}[1_A|\mathcal{F}^{K_n}]\underset{L^p(\mathbb{P})}{\overset{a.s.}{\rightarrow}}1_A.$$
Therefore,
$$\lim\limits_{n\rightarrow\infty}\mathbb{E}[1_A\mathbb{E}[1_A|\mathcal{F}^{K_n}]]=\mathbb{P}[A].$$
On the other hand, by independence between $\mathcal{F}^{K_n}$ and $\bigcap\limits_{K\text{ finite}}\mathcal{F}_{K^c}$,
$$\mathbb{E}[1_A\mathbb{E}[1_A|\mathcal{F}^{K_n}]]=\mathbb{P}[A]\mathbb{P}[\mathbb{P}[A|\mathcal{F}^{K_n}]]=(\mathbb{P}[A])^2.$$
The $0$-$1$ law holds since $\mathbb{P}[A]=(\mathbb{P}[A])^2$.
\end{proof}

\begin{cor}\label{cor: equivalence between theta and p general graph}
For any $x\in V$,
$\mathbb{P}[\#\mathcal C_{\alpha,\kappa}(x) = \infty]>0$ iff $\mathbb P[\cup_{z\in V}\{\#\mathcal C_{\alpha,\kappa}(z) = \infty\}]=1$.
\end{cor}
\begin{proof}
By the definition of Poisson loop soup, $(1_{\{\ell\in\mathcal{L}_{\alpha,\kappa}\}})_{\ell\in\{\text{loops}\}}$ 
is a sequence of $0,1$ valued independent variables. 
By the FKG inequality for product measures, for loop percolation associated with irreducible random walk, 
if $\mathbb{P}[\#\mathcal C_{\alpha,\kappa}(x)= \infty]>0$ for some $x\in V$, then 
such probability is positive for any vertex in $V$. 
Since $V$ is countable, if for a given $x\in V$, $\mathbb P[\#\mathcal C_{\alpha,\kappa}(x) = \infty] = 0$, then 
$\mathbb P[\cup_{z\in V}\{\#\mathcal C_{\alpha,\kappa}(z) = \infty\}] = 0$. 
The inverse statement follows from Proposition~\ref{prop:zero-one law for tail event}. 
\end{proof}

\medskip

The next statement shows that $\kappa_c(\alpha) \leq 0$ for any recurrent graph. 
\begin{prop}\label{prop: existence of percolation for recurrent graph and zero kappa}
For a connected recurrent graph $G=(V,E)$, $\alpha>0$, $\kappa = 0$, and $x\in V$, 
$\mathbb P$ almost surely, the set $\{\ell\in\mathcal{L}_{\alpha,0}~:~x\in \ell\}$ of loops passing through $x$ covers $V$.
\end{prop}
\begin{proof}
Let $X_n$ be a SRW on $G$. 
For any different $x,y\in V$, let $\tau_0 = 0$, $\tau_{2k+1} = \inf \{n>\tau_{2k}~:~X_n = x\}$, and $\tau_{2k+2} = \inf\{n>\tau_{2k+1}~:~ X_n = y\}$ for $k\geq 0$. 
Since $G$ is recurrent, $\tau_i$ are all finite almost surely for any initial position of the random walk. 
By \eqref{def:dotmu} and \eqref{eq:muS1S2}, 
\[
\mu(\ell~:~x\stackrel{\ell}\longleftrightarrow y) = \sum_{k\geq 1}\frac 1k\cdot \mathbb P^x[\tau_{2k}<\infty] = \sum_{k\geq 1}\frac 1k = \infty.
\]
Therefore, 
\[
\mathbb P[\exists \ell\in \mathcal L_{\alpha,0}~:~x\stackrel{\ell}\longleftrightarrow y] = 
1 - \exp\left\{-\alpha\mu(\ell~:~x\stackrel{\ell}\longleftrightarrow y)\right\} = 1. 
\]
Since $G$ is connected and $V$ is countable, it follows that $\mathbb P$ almost surely $G$ is covered by the loops from $\mathcal L_{\alpha,0}$ which visit $x$. 
\end{proof}

\medskip

Next, we study the continuity of the critical curve $(\alpha,\kappa)$ of the loop percolation model. 
By \cite[Proposition~4.3]{LemaireLeJan}, both $\thr(\kappa)$ and $\kappa_c(\alpha)$ are non-decreasing. 
It follows from Theorem~\ref{thm:phase transition} that $\kappa\mapsto \thr(\kappa)$ is discontinuous at $\kappa = 0$ for loop percolation on $\mathbb Z^d$, $d\geq 3$. 
We do not know if it is continuous for $\kappa>0$. 
Our following result is about the continuity of $\alpha\mapsto\kappa_c(\alpha)$.
\begin{prop}\label{Continuity of kappa_c(alpha)} 
For $\alpha>0$, $\kappa_c(\alpha)$ is a non-decreasing continuous function of $\alpha$.
\end{prop}

For its proof, we need the following lemma.
\begin{lem}\label{lem:compare p(alpha,kappa)}
For $\kappa_1\geq\kappa_0$ and $\alpha_1\leq \alpha_0\left(\frac{1+\kappa_1}{1+\kappa_0}\right)^2$, 
the point process of loops $\mathcal L_{\alpha_1,\kappa_1}$ is stochastically dominated by $\mathcal L_{\alpha_0,\kappa_0}$, 
i.e., there exists a coupling $(\mathcal L_0,\mathcal L_1)$ of $\mathcal L_{\alpha_0,\kappa_0}$ and $\mathcal L_{\alpha_1,\kappa_1}$ 
on some probability space $(\widehat \Omega, \widehat {\mathcal A}, \widehat {\mathbb P})$ 
such that $\widehat {\mathbb P}[\mathcal L_1\subseteq \mathcal L_0 ] = 1$. 
\end{lem}
\begin{proof}
By definition, $\alpha_0\mu_{\kappa_0}$ is the intensity measure of the Poisson loop ensemble $\mathcal{L}_{\alpha_0,\kappa_0}$, 
and $\alpha_1\mu_{\kappa_1}$ is the intensity measure of $\mathcal{L}_{\alpha_1,\kappa_1}$. Under the assumption of this lemma, for a fixed based loop $(x_1,\ldots,x_k)$,
\begin{align}\label{eq:lcp1}
\alpha_0\dot\mu_{\kappa_0}((x_1,\ldots,x_k))=&\alpha_0\frac{1}{(1+\kappa_0)^k}\frac{1}{k}Q^{x_1}_{x_2}\cdots Q^{x_{k-1}}_{x_k}Q^{x_k}_{x_1}\notag\\
=&\alpha_0\left(\frac{1+\kappa_1}{1+\kappa_0}\right)^k\frac{1}{k}\frac{1}{(1+\kappa_1)^k}Q^{x_1}_{x_2}\cdots Q^{x_{k-1}}_{x_k}Q^{x_k}_{x_1}\notag\\
=&\alpha_0\left(\frac{1+\kappa_1}{1+\kappa_0}\right)^k\dot\mu_{\kappa_1}((x_1,\ldots,x_k))\notag\\
\geq &\alpha_0\left(\frac{1+\kappa_1}{1+\kappa_0}\right)^2\dot\mu_{\kappa_1}((x_1,\ldots,x_k))\notag\\
\geq &\alpha_1\dot\mu_{\kappa_1}((x_1,\ldots, x_k)).
\end{align}
For two measures $\nu_1$ and $\nu_2$, we write $\nu_1\leq\nu_2$ iff $\nu_1(A)\leq \nu_2(A)$ for any measurable $A$. 
Then, \eqref{eq:lcp1} implies that $\alpha_0\mu_{\kappa_0}\geq \alpha_1\mu_{\kappa_1}$, 
and the existence of the claimed coupling of $\mathcal L_{\alpha_0,\kappa_0}$ and $\mathcal L_{\alpha_1,\kappa_1}$ follows. 
\end{proof}

\begin{proof}[Proof of Proposition \ref{Continuity of kappa_c(alpha)}]
Fix $0<\alpha_0<\alpha_1$. 
By Lemma~\ref{lem:compare p(alpha,kappa)}, for any $\kappa_1<\kappa_c(\alpha_1)$, $\kappa_0 = \sqrt{\frac{\alpha_0}{\alpha_1}}(1+\kappa_1) - 1\leq \kappa_c(\alpha_0)$. 
Thus, 
\[
0\leq \kappa_c(\alpha_1) - \kappa_c(\alpha_0)\leq \frac{1 + \kappa_c(\alpha_1)}{\sqrt{\alpha_1}}\cdot \left(\sqrt{\alpha_1} - \sqrt{\alpha_0}\right)
\leq \frac{1 + \kappa_c(\alpha_0)}{2\alpha_0}\cdot \left(\alpha_1 - \alpha_0\right) .\qedhere
\]
\end{proof}

\section{Open problems}
\begin{itemize}\itemsep0pt
\item For $d=3$, does the limit $\lim\limits_{n\rightarrow\infty}\frac{1}{\log n}\log\mathbb{P}[0\overset{\mathcal{L}_{\alpha}}{\longleftrightarrow}\partial B(0,n)]$ exist? 
Note that if the limit exists, it must depend on $\alpha$ by Theorems \ref{thm: polynomial decay} and \ref{thm: d-2 is not the right exponent for d=3}. 
Assume that the limit exists and equals $-c(\alpha)$. Does the limit $\lim\limits_{n\rightarrow\infty}n^{c(\alpha)}\mathbb{P}[0\overset{\mathcal{L}_{\alpha}}{\longleftrightarrow}\partial B(0,n)]$ exist?
What is the value of $c(\alpha)$?
Conditionally on $0\overset{\mathcal{L}_{\alpha}}{\longleftrightarrow}\partial B(0,n)$, is the typical loop distance between $0$ and $\partial B(0,n)$ of order $\log n$? 

\item For $d=4$, does $\mathbb{P}[0\overset{\mathcal{L}_{\alpha}}{\longleftrightarrow}\partial B(0,n)]$ equal $n^{-2}$ up to some multiplicative factor of power of $\log n$? 
If so, does the power of the logarithm depend on $\alpha$? 
Conditionally on $0\overset{\mathcal{L}_{\alpha}}{\longleftrightarrow}\partial B(0,n)$, is the typical loop distance between $0$ and $\partial B(0,n)$ of order $\log\log n$? 

\item For $d\geq 5$, do the limits $\lim\limits_{n\rightarrow\infty}n^{d-2}\mathbb{P}[0\overset{\mathcal{L}_{\alpha}}{\longleftrightarrow}\partial B(0,n)]$, 
$\lim\limits_{||x||\rightarrow\infty}||x||^{2(d-2)}\mathbb{P}[0\overset{\mathcal{L}_{\alpha}}{\longleftrightarrow}x]$, 
$\lim\limits_{n\rightarrow\infty}n^{\frac d2 - 1}\mathbb{P}[\#\mathcal C_{\alpha}(0) > n]$ exist?
\item For $d\geq 3$, is $\thran = \thr$?
\item For $d=5$, is $\thrcz=\thr$? By Proposition~\ref{prop: alpha * is smaller than alpha **} this would imply also that $\thran = \thr$. 
\end{itemize}

\paragraph{Acknowledgements.} 
We thank Yves Le Jan, Sophie Lemaire and Titus Lupu for useful suggestions and stimulating discussions.

\bibliographystyle{amsalpha}
\providecommand{\bysame}{\leavevmode\hbox to3em{\hrulefill}\thinspace}
\providecommand{\MR}{\relax\ifhmode\unskip\space\fi MR }
\providecommand{\MRhref}[2]{%
  \href{http://www.ams.org/mathscinet-getitem?mr=#1}{#2}
}
\providecommand{\href}[2]{#2}

\end{document}